\documentclass[a4paper,10pt,intlimits,sumlimits]{amsart}
\usepackage{paper}
\usepackage{macros}
\usepackage{texmacs}
\newcommand{\Lap}{\mathop{\kern0pt \mathscr{L}}\mathopen{}} 
\newcommand{\xs}{\mf{s}} 
\newcommand{\Hess}{\mathop{\kern0pt \mathrm{D}^2}\mathopen{}} 
\newcommand{\LP}{\mathop{\kern0pt \mathrm{P}}\hspace{-0.20em}\mathopen{}}
\newcommand{\QP}{\mathop{\kern0pt \mathrm{Q}}\mathopen{}} 
\newcommand{\UP}{\mathop{\kern0pt \mathrm{U}}\hspace{-0.15em}\mathopen{}}

\begin{document}

\title{Almost sure local well-posedness for cubic nonlinear Schrödinger
equation with higher order operators}
 
\author{Jean-Baptiste Casteras}
\address{CMAFcIO, Faculdade de Ciências da Universidade de Lisboa, Edificio
C6, Piso 1, Campo Grande 1749-016 Lisboa, Portugal}
\email{jeanbaptiste.casteras@gmail.com}
\tmnote{J.-B.C. supported by FCT --- Fundação para a Ciência e a Tecnologia,
under the project: UIDB/04561/2020}
\providecommand{\todoJBC}[1]{\todo[color=brown!40]{#1}} 

\author{Juraj Földes}
\address{Dept. of Mathematics, University of Virginia, Kerchof Hall,
Charlottesville, VA 22904-4137}
\email{foldes@virginia.edu}
\tmnote{J. F. was partly supported by grant NSF-DMS-1816408}
\providecommand{\todoJF}[1]{\todo[color=green!40]{#1}} 

\author{Gennady Uraltsev}
\address{Dept. of Mathematics, University of Virginia, Kerchof Hall,
Charlottesville, VA 22904-4137}
\email{gennady.uraltsev@gmail.com}
\providecommand{\todoGU}[1]{\todo[color=red!40]{#1}} 

\renewcommand{\shorttitle}{Local well-posedness for cubic NLS}

\begin{abstract}
In this paper, we study the local well-posedness of the cubic Schrödinger
equation:
\[
(i \partial_t - \Lap) u = \pm |u|^2 u \quad \text{on } I \times \R^d,
\]
with randomized initial data, and $\Lap$ being an operator of degree $\sigma
\geq 2$. Using estimates in directional spaces, we improve and
extend known results for the standard Schrödinger equation (i.e. $\Lap =  \Delta$) to any dimension and obtain results under natural assumptions for general $\Lap$.
\end{abstract}



{\maketitle}

\section{Introduction}

In this paper, we investigate the local well-posedness of the cubic
nonlinear Schrödinger equation
\begin{equation}\label{eq:cubic-NLS-intro}
\begin{cases}
(i \partial_t - \Lap) u = \pm |u|^2u & \text{ on } I \times \R^d,
\\
u (0) = f \in H_x^S (\R^d)&
\end{cases}
\end{equation}
with a general operator $\Lap$ and randomized initial conditions, see
\eqref{eq:randomization} below.

First, we illustrate our results for the classical cubic Schrödinger
equation, that is, for \eqref{eq:cubic-NLS-intro} with $\Lap = - \Delta$.

\begin{theorem}
\label{thm:main-special}
Fix $d \geq 3$ and 
\[
S > \frac{d - 2}{2} \times \begin{cases}
\frac{1}{3} & \text{ if } d = 3 \,,
\\
\frac{d - 3}{d - 1} & \text{ if } d \geq 4 \,,
\end{cases}
\]
and assume $f \in H_x^S(\R^d)$. If $f^\omega$ is the randomization of
$f$ as in \eqref{eq:randomization}, then almost surely there exists an
open interval $0 \in I$ and a unique solution
\[ u (t) \in e^{i t \Delta} f^{\omega} + 
C (I ; H^{\frac{d - 2}{2}} (\R^d))
\]
to 
\begin{equation}\label{eq:cubic-NLS-intro-1}
\begin{cases}
(i \partial_t + \Delta) u = \pm |u|^2u & \text{ on } I \times \R^d,
\\
u (0) = f^\omega \in H_x^S (\R^d)\,. &
\end{cases}
\end{equation}

\end{theorem}

As detailed below, \Cref{thm:main-special} improves known
results in dimensions \(d\geq 5 \); furthermore, our general
\Cref{thm:main} extends \Cref{thm:main-special} to a large
class of operators $\Lap$, improving on existing results in all
dimensions.  We begin by briefly reviewing background and known
results for \eqref{eq:cubic-NLS-intro} with fixed deterministic
initial condition $f$. The main operators of interest are
$\Lap = -\Delta$ yielding classical Schrödinger equation, and
$\Lap = \Delta^2 - \mu \Delta$, $\mu \in \{- 1, 0, 1\}$ leading to the fourth order
Schrödinger equation with mixed dispersion introduced by Karpman and
Shagalov {\cite{karpmanStabilitySolitonsDescribed2000}} (see also
{\cite{karpmanStabilizationSolitonInstabilities1996}}).

Let
$S_{\mrm{crit}} \eqd \frac{d - \sigma}{2}$ be the special value called the energy critical
exponent; for initial data $f \in H^S (\R^d)$, we say that the Cauchy
problem \eqref{eq:cubic-NLS-intro} is
\[ 
\begin{cases}
\text{subcritical}   & \text{ if } S > S_{\mrm{crit}},     \\
\text{critical}      & \text{ if } S = S_{\mrm{crit}},     \\
\text{supercritical} & \text{ if } S < S_{\mrm{crit}} \,.
\end{cases} 
\]
The relevance of $S_{\mrm{crit}}$ can be seen by neglecting lower order
terms, that is, by assuming that $\Lap = (- \Delta)^{\sigma / 2}$. Then,
\eqref{eq:cubic-NLS-intro} possesses a natural scaling symmetry: if
$u$ satisfies the equation in \eqref{eq:cubic-NLS-intro}, then
\[
u_{\lambda} (t, x) = \lambda^{\frac{\sigma}{2}} u (\lambda^{\sigma} t, \lambda x), \qquad \lambda > 0,
\]
also satisfies the same equation. In addition,
\[
\|u_{\lambda} (0)\|_{\dot{H}^S} = \lambda^{S - S_{\mrm{crit}}} \|u (0)\|_{\dot{H}^S} = \lambda^{S  - S_{\mrm{crit}}} \|f\|_{\dot{H}^S} \,,
\]
where $\| \cdot \|_{\dot{H}^S}$ denotes the homogeneous Sobolev norm (see
\Cref{sec:notation} for the definition).

In the subcritical or critical regime, local in time solutions can be
constructed using Strichartz estimates and a classical fixed point
argument.  We refer to
{\cite{cazenaveSemilinearSchrodingerEquations2003,cazenaveCauchyProblemCritical1990,collianderGlobalWellposednessScattering2008,ryckmanGlobalWellposednessScattering2007}}
for results concerning the classical nonlinear Schrödinger equation
and to
{\cite{pausaderGlobalWellposednessEnergy2007,pausaderMasscriticalFourthorderSchrodinger2010}}
for results on the fourth order one. On the other hand, in the
supercritical regime, \eqref{eq:cubic-NLS-intro} is ill-posed by a
result of Christ, Colliander, and Tao
{\cite{christIllposednessNonlinearSchrodinger2003}}.

From a practical perspective, ill-posedness is observable only if it
does not vanish after an introduction of small fluctuations. Unlike
the Schrödinger equation, the initial conditions often originate in
measurements, which are naturally susceptible to errors that are
inherently random.  We use a standard randomization of initial data
based on a unit-scale decomposition of frequency space. The unit scale
in frequency can be thought of as characteristic scale of our
measurements or of the experiment at hand.  Other randomization
approaches have been investigated in
{\cite{burqLongTimeDynamics2013,dengTwodimensionalNonlinearSchrodinger2012,thomannRandomDataCauchy2009,spitzAlmostSureLocal2023}}.
The scale is characterized by the $\psi \in C_c^{\infty} (\R^d)$, which is an
even, non-negative cut-off function supported in the unit-ball of
$\R^d$ centered at $0$, such that
\[
\sum_{k \in \Z^d} \psi (\xi - k) = 1
\]
for all $\xi \in \R^d$. Then, for $f \in H^S_x (\R^d)$, we define the
randomization of $f$ by
\begin{equation}\label{eq:randomization}
f^{\omega} = \sum_{k \in \Z^d} g_k (\omega) \QP_{k} f\,. 
\end{equation}
Here $(g_k)_{k \in \Z^d}$ is a sequence of i.i.d zero-mean complex random
variables with finite moments of all orders on a probability space
$(\Omega, \mathcal{A}, \mathbb{P})$ (for example \(\mathcal{N}(0,1;\C)\) Gaussian variables). 
The operators $\QP_{k}$ are unit scale frequency approximate projection operators given on the frequency side by 
\begin{equation}\label{eq:proj-unit-scale}
\Fourier (\QP_{k} f) (\xi) = \psi (\xi - k)\Fourier(f) (\xi),\qquad \text{for} \quad \xi \in \R^d,
\end{equation}
where $\Fourier (f)$ stands for the Fourier transform of $f$. 

The randomization \eqref{eq:randomization} does not improve the
differentiability properties of $f$: if
$f \in H^S (\R^d) \setminus H^{S + \epsilon} (\R^d)$ for some
$\epsilon > 0$, then
$f^{\omega} \in H^S (\R^d) \setminus H^{S + \epsilon} (\R^d)$ almost surely, see
{\cite{burqRandomDataCauchy2008a}}. In particular, if the problem was
super-critical for $f$, then it stays almost surely super-critical for
the initial condition $f^\omega$. However, since $\QP_{k} f$ is localized on a
set of bounded diameter in Fourier space, the Bernstein inequality
(see Lemma \ref{lem:bernstein} below) implies for any
$1 \leq r_1 \leq r_2 \leq \infty$ and any $k \in \Z^d$ that
\begin{equation}\label{eq:unit-scale-bernstein}
\|\QP_{k} f\|_{L^{r_2}_x (\R^d)} \leq C_{r_1,r_2} \|\QP_{k} f\|_{L^{r_1}_x (\R^d)}
\end{equation}
with a constant $C$ independent of $k$. We exploit \eqref{eq:unit-scale-bernstein} to show that $f^\omega$ and
$e^{- i t \Lap} f^{\omega}$ possess better local integrability properties
of than $f$ and $e^{- i t \Lap} f$, respectively (see \Cref{lem:proba1}). 

The first results on the probabilistic well-posedness were proved by
Bourgain
{\cite{bourgainPeriodicNonlinearSchrodinger1994,bourgainInvariantMeasuresDdefocusing1996}}
and McKean \cite{mckeanStatisticalMechanicsNonlinear1995}, who showed that a suitable randomization
of the initial data can be used to construct local or even global
solutions in the supercritical regime.  Specifically, they proved an
almost sure local existence of solutions of
\eqref{eq:cubic-NLS-intro-1} on torus. Also, with a help of invariant
(Gibbs) measures the local solutions were extended to global ones (see
also
{\cite{lebowitzStatisticalMechanicsNonlinear1988,syAlmostSureGlobal2021}}
for other results in this direction). There is a number of results for
the randomized nonlinear Schrödinger equation on torus, (see survey
\cite{nahmodNonlinearSchrodingerEquation2015} and references therein) but the techniques are very
different compared to the non-compact case of $\R^d$.  For example,
local smoothing obtained in \Cref{lem:dir-local-smoothing} is not
expected to hold true on torus. Randomization techniques on the torus
were used for other equations such as Navier-Stokes equation
\cite{zhangRandomDataCauchy2012}, nonlinear wave equation
\cite{burqRandomDataCauchy2008}, or Hartree NLS
\cite{dengInvariantGibbsMeasure2021}.

The literature contains several well-posedness results for
\eqref{eq:cubic-NLS-intro-1} on $\R^d$ using our randomization. Bényi,
Oh, and Pocovnicu {\cite{benyiProbabilisticCauchyTheory2015}} proved
the almost sure local well-posedness of \eqref{eq:cubic-NLS-intro} for
$d \geq 3$ and $S > \frac{d - 1}{d + 1} \frac{d - 2}{2}$ in the following sense:
there exist $c, C, \gamma > 0$ such that for each $0 < T \ll 1$, there exists
a set $\Omega_T \subset \Omega$ with the following properties :
\begin{itemize}
\item $\mathbb{P} (\Omega_T^c) < \mrm{Cexp} (- \frac{c}{T^{\gamma}
      \|f\|_{H^S}^2})$.

\item For almost each $\omega \in \Omega_T$, there exists a unique solution u to
      \eqref{eq:cubic-NLS-intro} with $u (0) = f^{\omega}$ in the class
      \[ e^{- i t \Delta} f^{\omega} + C ([- T, T] ; H^{S_{\mrm{crit}}} (\R^d)) \subset C
      ([- T, T] ; H^S (\R^d)) . \]
\end{itemize}
Later, Brereton {\cite{breretonAlmostSureLocal2018}} obtained analogous
results for $\Lap = - \Delta$ and quintic non-linearity. When $d = 3$, Shen,
Soffer, and Wu {\cite{shenAlmostSureWellPosedness2023}} 
recently, obtained
the local well-posedness of \eqref{eq:cubic-NLS-intro-1}
for $S \geq \frac{1}{6}$ improving
{\cite{benyiProbabilisticCauchyTheory2015}}. All described results rely on
a fixed point argument for operators on variants of the $X^{s, b}$ spaces
adapted to the variation spaces $V^p$ and $U^p$ introduced by Koch, Tataru,
and collaborators
{\cite{hadacWellposednessScatteringKPII2009,herrGlobalWellposednessEnergycritical2011,kochDispersiveEquationsNonlinear2014}}.
The result of {\cite{benyiProbabilisticCauchyTheory2015}} was also improved by
Dodson, Lührmann, and Mendelson {\cite{dodsonAlmostSureLocal2019}} when $d =
4$, which corresponds to the energy-critical Schrödinger equation. More
precisely, they proved the local well-posedness of \eqref{eq:cubic-NLS-intro-1}
when  $d = 4$ and $S > \frac{1}{3}$. It is important to
note, that instead of using variants of $X^{s, b}$,
{\cite{dodsonAlmostSureLocal2019}} used a directional norm denoted by $L_e^{a,
b}$, $e \in S^{d - 1} \subset \R^d$, introduced by Ionescu and Kenig
{\cite{ionescuLowregularitySchrodingerMaps2006,ionescuLowregularitySchrodingerMaps2007}}
to prove well-posedness for the Schrödinger map equation.

The only local well-posedness result for
general \eqref{eq:cubic-NLS-intro} with a 
higher order operators $\Lap$ and
randomized data was obtained in {\cite{duongdinhRandomDataTheory2021}} for
$\Lap = | \Delta |^2 - \mu \Delta$, $\mu \geq 0$, $d \geq 5$ under the assumption $S > \max
\Big\{ \frac{(d - 1) (d - 4)}{2 (d + 5)}, \frac{d - 4}{4} \Big\}$.

The idea of the mentioned results is to subtract the linear evolution
of the initial condition given by $e^{- i t \Delta} f^{\omega}$ which presumably
has the worst regularity. Then, the regularity level $S$ is chosen
such that the reminder is smooth enough to belong to a sub-critical
space, where the fixed point argument can be used.

By using iterative procedure based on a partial power expansion, one
can subtract higher order terms as in Bényi, Oh, and Pocovnicu
{\cite{benyiHigherOrderExpansions2019}} to obtain local well-posedness
for any $S > \frac{1}{6}$. Note that the condition $S > \frac{1}{6}$ was
improved (by including the endpoint) to $S \geq \frac{1}{6}$ in the mentioned
result \cite{shenAlmostSureWellPosedness2023} without need for
iterations.

\Cref{thm:main-special}, the special case of our general
\Cref{thm:main}, recovers or improves upon the existing results in any
dimension with a unified approach. In particular, we obtain the
optimal condition from {\cite{dodsonAlmostSureLocal2019}} for $d = 4$,
we reproduce the result in {\cite{shenAlmostSureWellPosedness2023}}
when $d = 3$ except for the endpoint regularity case, and we improve
{\cite{benyiProbabilisticCauchyTheory2015}} (or any other existing
result) for $d \geq 5$. In addition, \Cref{thm:main-special} allows for
generalizations to other operators, as detailed below, and in a
forthcoming work \cite{CFU23} we show how to use our framework to
further lower the regularity requirements by including higher order
expansions. We remark that our general \Cref{thm:main} improves
all existing results for higher order operators in any dimensions (see
for example \cite{duongdinhRandomDataTheory2021}).

Next, we specify our
assumptions on $\Lap$ and formulate our general result.  We abuse
notation and denote by $\Lap$ both the differential operator and its
symbol so that
\[ 
\Lap f (x) = \int_{\R^d} e^{2 \pi i x \xi} \Lap (2 \pi i \xi) \FT{f} (\xi) \dd \xi.
\]
We assume that the symbol is real-valued, and there is a real
$\sigma \geq 2$ such that for all $\xi$ large enough one has (assumptions are on
the symbol)
\begin{align} \label{eq:symbol-bddness}
&
\begin{aligned}
& |\partial^{\alpha}\Lap(\xi)|\lesssim | \xi |^{\sigma - |\alpha|}\quad\textrm{ for all } |\alpha|\leq\Big\lfloor \frac{d}{\sigma} \Big\rfloor+2,
\\ &  | \nabla \Lap (\xi) |\gtrsim | \xi |^{\sigma - 1}
\end{aligned} 
\\ \label{eq:symbol-curvature} &
| \xi |^{d (\sigma - 2)} \lesssim \big| \det \Hess \Lap (\xi) \big|  \lesssim | \xi |^{d (\sigma - 2)} \,, 
\end{align}
where $\Hess \Lap (\xi)$ is the Hessian of $\Lap (\xi)$ and
$\lfloor z \rfloor$ denotes the integer part of $z$, that is, the largest integer
smaller than $z$. These conditions are trivially satisfied if
$\Lap = | \Delta |^{\sigma / 2} + \Lap^{\sharp}$ for $\sigma \geq 2$ and
$\Lap^{\sharp}$ is a lower order operator with real, smooth symbol.

For the differentiability $S$ of initial data, we show, quite
interestingly, that there are two different regimes:
\begin{equation}\label{eq:Smin}
S_{\min} (\sigma, d) \eqd \frac{d - \sigma}{2} \times 
\begin{cases}
\frac{1}{3}                    & \text{ if } \frac{d + 2}{3} \leq \sigma,  \\
\frac{d + 1 - 2 \sigma}{d - 1} & \text{ if } \sigma \leq \frac{d + 2}{3} .
\end{cases} 
\end{equation}

Our main result reads as follows:

\begin{theorem}\label{thm:main}
Let $d > \sigma \geq 2$, and let $\Lap$ be a differential
operator whose Fourier symbol $\Lap$ is smooth, real, and satisfies
\eqref{eq:symbol-bddness} and \eqref{eq:symbol-curvature}. For any
$S > S_{\min} (\sigma, d)$ with $S_{\min}$ as in \eqref{eq:Smin} assume
$f \in H_x^S (\R^d)$ and let $f^\omega$ be the randomization of $f$ as in
\eqref{eq:randomization}.  Then, for a.e. $\omega \in \Omega$, there exists an
open interval $0 \in I$ and a unique solution
\[ 
  u (t) \in e^{- i t \Lap} f^{\omega} + C (I ; H^{S_{\mrm{crit}}} (\R^d))
\]
to
\begin{equation} \label{eq:eqintro}
\begin{cases}
(i \partial_t - \Lap) u = \pm |u|^2 u & \text{ on } I \times \R^d \\
u (0) = f^{\omega} .                  &
\end{cases} 
\end{equation}
\end{theorem}

The special case $\Lap = -\Delta$ was formulated in Theorem
\ref{thm:main-special} and when $\Lap = \Delta^2 \pm \mu \Delta$, we obtain local
well-posedness if
\[ S_{\min} (4, d) = \frac{d - 4}{2} \times \begin{cases}
\frac{1}{3}         & \text{ if } 4 < d \leq 10, \\
\frac{d - 7}{d - 1} & \text{ if } d \geq 10.
\end{cases} \]

Let us briefly comment on assumptions of Theorem \ref{thm:main}.  The
condition $d > \sigma$ ensures that we are in the super-critical regime (we
do not consider $S < 0$, because there the functions in $H^S$ are not
defined point-wise and the non-linearly $|u|^2u$ has to be interpreted
differently).  If $d \leq \sigma$, that corresponds to any \(S>0\) being energy
subcritical, the well-posedness results can be obtained using softer
techniques. Our assumptions on $\Lap$ are satisfied, for example, by
the operator $\Lap = \Delta^2 \pm \Delta$, or by symbols that may fail to be
convex. Furthermore, our conditions are stable with respect to
perturbations by lower-order terms.

A careful inspection of our methods gives explicit estimates from
below on the time of existence of solutions, similar to ones mentioned
above and in \cite{benyiProbabilisticCauchyTheory2015}. There are
several techniques that extend the local well-posedness theory to
global well-posedness with high probability for small initial
data. These techniques and results are closely related to
scattering. There are, however, obstacles to such approaches for the
class of operators $\Lap$ we consider. First, since we work in
the energy supercritical regime, solutions do not satisfy a priori
global energy estimates.  Second, there are no global in time
dispersive estimates even for the linear evolution.  In particular, if
the symbol of $\Lap$ has vanishing curvature on non-trivial sets
scattering behavior is unlikely.  Overall, for the clarity and length
of the manuscript we decided not to include global-in-time existence
results, which would require us to restrict the class of operators
considered. We also omit these explicit considerations on the time of
existence for local solutions.

The proof of \Cref{thm:main} starts by subtracting the free (random)
evolution $F = e^{- i t \Lap} f^{\omega}$, which, heuristically, has the
lowest regularity. This allows us to transform \eqref{eq:eqintro} into
a forced cubic NLS equation for the remainder term $v=u-F$:
\begin{equation}\label{eq:NLS-forced}
\begin{cases}
(i \partial_t - \Lap) v = \pm |F + v|^2 (F + v) := H(t, x), \\
v (0) = 0 \,.
\end{cases} 
\end{equation}
Due to stochastic cancellation effects, $F$ has almost surely better space-time
integrability properties (but not smoothness) then its
deterministic counter-part. In fact, the problem becomes sub-critical
allowing one to find $v$ by finding the fixed point by the Banach fixed point theorem.

Inspired by a functional framework of
{\cite{dodsonAlmostSureLocal2019}}, we use directional spaces and
prove the contraction of an appropriate map in a sub-critical space
denoted by $X^{\xs + \tilde{\epsilon}, \epsilon}$. The space
$X^{\xs + \tilde{\epsilon}, \epsilon}$ contains classical Strichartz, directional maximal
type, and directional local smoothing type components. The central
idea relies on the observation, that if
$v \in X^{\xs + \tilde{\epsilon}, \epsilon}$, then the forcing term
$H$ belongs to the dual space
$(X^{\xs + \tilde{\epsilon}, \epsilon})^*$. This differs from the approach of
\cite{dodsonAlmostSureLocal2019}, where the local smoothing term is
absent from the norm of the space
$X^{\xs + \tilde{\epsilon}, \epsilon}$, and $H$ needs to be controlled in an
space denoted $G$ (see \cite{dodsonAlmostSureLocal2019}). Controlling
the forcing term $H$, in a dual space of
$X^{\xs + \tilde{\epsilon}, \epsilon}$ rather than $G$ proves more natural and appropriate for
generalizations to higher dimensions. Finally, the extension to more
general operators $\Lap$ requires a finer analysis of the oscillatory
integrals occurring in the study of the free evolution of the
Schrödinger equation; some of these arguments have a micro-local
flavor (see \Cref{lem:dir-local-smoothing}).

The proof of \Cref{thm:main} can be summarized by the following steps.

\begin{description}
\item[Step 1] We control the linear evolution $F$ in the norm of the
space $Y^{S, \epsilon}$, which has regularity $S > S_{\min}$, and in
particular $Y^{S, \epsilon}$ is super-critical. The norm of
$Y^{S, \epsilon}$ is based on a dyadic decomposition of Fourier space and on
each dyadic annulus in frequency, the norm of $Y^S$ is an
appropriately weighted combination of classical space-time norms
$L_t^p L_x^q$ and directional norms $L_{e_l}^{a, b}$. The
cancellations stemming from the randomization of initial conditions,
allow us to use a higher integrability exponent in the spatial
directions compared to the classical Strichartz norms.

\item[Step 2] The existence and uniqueness of $v$ is showed by the
fixed point argument in the space $X^{\xs, \epsilon} (I)$. The parameter
$\xs \gtrsim S_{\mrm{crit}} = \frac{d - \sigma}{2}$ indicates differentiability and
$\epsilon > 0$ is a small parameter that allow us to avoid working in
endpoint spaces such as $L^\infty$. The space $X^{\xs, \epsilon}$ is endowed with
a norm based, again, on a dyadic composition of Fourier space; on each
dyadic annulus in frequency, the norm is a combination of Strichartz
admissible $L_t^p L_x^q$ norms and appropriately weighed directional
norms $L_{e_l}^{a, b}$.  The key step is the estimate of $v$ in the
$X^{\xs, \epsilon} (I)$ norm, by the
$(X^{\xs + \tilde{\epsilon}, \epsilon})^*$ norm of $H$ (see \eqref{eq:NLS-forced} for the
definition of $H$).

\item[Step 3] Based on the results described in Step 2, one needs to
control the forcing term $H$ in the
$(X^{\xs + \tilde{\epsilon}, \epsilon})^*$ norm.  More specifically, if
$F \in Y^{S, \epsilon}$ (established in Step 1) and
$v \in X^{\xs, \epsilon}$ (postulated in Step 2), then we show that
$H$ belongs to $X^{\xs+ \tilde{\epsilon}, \epsilon} (I)^*$. Since
$H$ can be viewed as sum of cubic monomials $\mathcal{C}$ in the variables
$F,v,\bar{F},\bar{v}$, controlling any $\mathcal{C}$ in the norm
$\big(X^{\xs+ \tilde{\epsilon}, \epsilon} (I)\big)^*$ requires, by duality, testing
$\mathcal{C}$ against a function with bounded $X^{\xs, \epsilon}$ norm. Thus, the
estimate on $\mc{C}$ is equivalent to suitable quadrilinear estimates,
which we factor through two bilinear estimates mapping into
$L^2_x L^2_t$.  Our bilinear estimates implicitly contain the bilinear
Strichartz estimates of \cite{bourgainRefinementsStrichartzInequality1998, ozawaSpacetimeEstimatesNull1998}. In
this step, we differ from the functional framework of
\cite{dodsonAlmostSureLocal2019}, where a different space appears
instead of $(X^{\xs + \tilde{\epsilon}, \epsilon})^*$, making generalizations less
efficient.
\end{description}

\begin{remark}
As observed in \cite{benyiHigherOrderExpansions2019}, one can 
use higher order multilinear expansions to obtain 
a solution to \eqref{eq:eqintro}. More precisely, one
can consider solutions to \eqref{eq:eqintro} of the form
\[
u=F_{1}+F_{3}+F_{5}+\ldots+F_{2k+1} + v_{k}
\]
with \(v\in C (I ; \dot{H}_x^{\alpha} (\R^d))\) for some $\alpha > \frac{d-\sigma}{2}$, where \(F_{1}=e^{-it\Delta}f^{\omega}\) and
\[
F_{2k+1}\eqd - i \int_0^t e^{- i (t - t') \Lap} |F_{1}+\ldots+F_{2k-1}|^{2} (F_{1}+\ldots+F_{2k-1}) \dd t'.
\]
The functions $F_j$ are chosen so that they cancel out higher order terms 
on the right-hand side of \eqref{eq:cubic-NLS-intro}, which 
are independent of $v$ and $\bar{v}$. 

Our forthcoming paper \cite{CFU23} uses such higher order expansion for \(\Lap=\Delta\) in \(d\in\{3,4\}\) and 
significantly improves the requirements on $S$. Thus, the functional
framework developed in this paper partly serves as the foundation for and efficient treatment of the higher order expansions.

When \(d \gg \sigma \), that is,  when the nonlinearity is `very supercritical', then the monomials $\mathcal{C}$ including $v$ become dominant, and further expansion is less
obvious. 
\end{remark}

\begin{remark}
The \(X^{\xs+ \tilde{\epsilon}, \epsilon} (I)\) norm, used for the fixed point theorem, is a combination of directional as well as classical Strichartz norms. The classical Strichartz norms are included for convenience, because, with minor adjustments, the contraction only needs bounds in the directional norms. 
\end{remark}

\begin{remark}
Our approach crucially relies on gain of derivatives in the
directional local smoothing estimates (see
\Cref{lem:dir-local-smoothing}). These estimates are not expected to
hold on compact domains.
\end{remark}

\begin{remark}
Our results split naturally into a deterministic analysis of
\eqref{eq:cubic-NLS-intro} in directional norms and  
probabilistic estimates that yields improved bound on the
free evolution in the directional norms. 
The former analysis may be of
interest in the general theory of Schrödinger operators, outside
the stochastic setting. We remark that a directional analysis
of Schrödinger equation was done in 
\cite{bennettSharpKplaneStrichartz2018}, however, to our best
knowledge, the results were not applied to 
non-linear Schrödinger equation. 
\end{remark}

The paper is organized as follows.

In \Cref{sec:generalities}, we recall generalities of Fourier analysis
and the Littlewood-Paley theory (dyadic annuli decomposition in
Fourier space).  We state classical Strichartz estimates for free
evolution operator. Then we introduce directional norms $L^{a, b}_e$
and prove directional maximal and local smoothing estimates for the
free evolution operator.  In \Cref{sec:linear-nonhomogeneous}, we use
results of \Cref{sec:generalities} to obtain estimates on solutions of
\eqref{eq:NLS-forced} with a generic forcing term $H$.  The bounds
proved in \Cref{sec:linear-nonhomogeneous} are sufficient to prove all
ideas in Step 2, above.  In \Cref{sec:probabilistic}, we establish
probabilistic estimates for the linear evolution of the randomized
initial data. We also recall several properties of sums of Gaussian
random variables.  \Cref{sec:nonlinear} contains trilinear estimates
that control interactions in the cubic non-linearity.  Finally, in
\Cref{sec:almost-sure-well-posedness} we establish Theorem
\ref{thm:main} by using a fixed point argument.

\Cref{thm:main} requires $d > \sigma \geq 2$ which we implicitly assume
henceforth, and we will not state it explicitly in the statements below.

\subsection{Notation}\label{sec:notation}

\begin{itemize}
\item We define $\N = \{0, 1, \cdots\}$ the set of non-negative integers.
\item In a $d$-dimensional space $\R^d$, we denote
$\{e_1, \cdots, e_d\}$ to be the standard basis.
  \item The Fourier transform of the function $f : \R^d \to \R$ is
  \[ \Fourier(f) (\xi) = \FT{f} (\xi) = \int_{\R^d} f (x) e^{- 2 \pi i \xi x} \dd x \,, \]
  where the dimension $d$ is deduced from the context. By the Fourier
  inversion formula:
  \[ f (x) = \int_{\R^d} \FT{f} (\xi) e^{2 \pi i \xi x} \dd x \]
  
  \item For two expressions $G$ and $H$ we write $G \lesssim H$ if there
  exists a constant $C > 0$ depending only on the fixed parameters of
  the problem such that $G \leq C H$. In particular, we typically assume
  that $C$ is independent of $N$. If $C$ depends on a variable
  $\epsilon$, we use $G \lesssim_\epsilon H$.
  
  \item We write $G \approx H$, if $G \lesssim H$ and $H \lesssim G$.
  
  \item The symbol $O (\epsilon)$ stands for any function $[0, 1) \rightarrow \R$
  such that
  \[ | O (\epsilon) | \lesssim \epsilon \]
  for all $\epsilon \in (0, 1]$. The specific function denoted by $O
  (\epsilon)$ can change from line to line.
  
  \item The open ball with radius $r$ and center $x$ is denoted $B_r (x)$; if $x =
  0$ we simply write $B_r$. The dimension of the ball is to be understood from
  the context.
  
  \item For $p > 1$, $p'$ stands for the dual of $p$, that is,
  $\frac{1}{p'} + \frac{1}{p} = 1$. If $p = 0$, then we set $p' = \infty$.
  
  \item The Japanese bracket $\langle \cdot \rangle$ is defined as $\langle N
  \rangle \eqd (1 + N^2)^{\frac{1}{2}}$.
  
  \item For $S > 0$, we denote $\langle \Delta \rangle^{S / 2}$ the operator with the
  Fourier multiplier $(1 +|4 \pi^{2} \xi^{2}|^2)^{S/4}$, that is,
  $\Fourier (\langle \Delta \rangle^{S / 2} f) (\xi) = \langle 4 \pi^{2} \xi^{2}\rangle^{S/2} \FT{f}(\xi)$. Then,
  $H^S (\R^d)$ denotes the Sobolev space endowed with the semi-norm
  \[
  \|u\|_{H^S (\R^d)}^2 = \| \langle \Delta \rangle^{S / 2} u (x) \|_{L^2 (\R^d) } .
  \]
  
  \item We denote by $\1_A$ the characteristic function of a set $A$, that is, 
  $\1_A (x) = 1$ if $x \in A$ and $\1_A (x) = 0$ otherwise. In addition, if for example 
  $x > y$, then we write $\1_{x > y}$ to indicate the function that
  is equal to $1$ when $x > y$ and vanishes otherwise. The variable
  of the function is to be deduced from the context.
  
  \item We denote by $\spt f \eqd \tmop{cl}\big(\{x \in \R^d : |f(x)| > 0 \}\big)$ the support of the function $f$, where  $\tmop{cl}( A)$ denotes
  the closure of a set $A$. Similarly, we define $\spt \FT{f}$.
  \item The diameter of a set \(A\subset\R^{d}\) is \(\diam(A)\eqd\sup_{x,y\in A}|x-y|\)
\end{itemize}

\section{Generalities}\label{sec:generalities}

In this section, we recall Littlewood-Paley projections, Strichartz
estimates, and we review directional norms $L^{a, b}_e$ introduced by
Ionescu and Kenig
{\cite{ionescuLowregularitySchrodingerMaps2006,ionescuLowregularitySchrodingerMaps2007}}. Then,
we prove new maximal function estimate \eqref{eq:dir-maximal} and a
local smoothing estimate \eqref{eq:dir-local-smoothing} for
$L^{a, b}_e$. Below, the estimate \eqref{eq:dir-local-smoothing}
allows us to `gain' $\frac{\sigma - 1}{2}$ derivatives in our estimates of the
nonlinear terms.

We begin by defining the Littlewood-Paley projections $\LP_{N}$ for
$N \in 2^{\N}$. For a fixed smooth cutoff function
$\phi \in C^{\infty} (\R)$, that is, a function such that
$\phi (\xi) = 1$ for $| \xi | \leq 1$ and $\phi (\xi) = 0$ for
$| \xi | > 1 + 2^{- 100}$, we set
\begin{equation}\label{eq:LP-bump}
\phi_N (\xi) \eqd
\begin{cases}
\phi (\xi) & \text{ if } N = 2^0,
\\
\phi \big( \frac{\xi}{N} \big) - \phi \big( \frac{\xi}{N / 2} \big) & \text{ if } N \in 2^n, n \in \N \setminus \{ 0 \} \, .
\end{cases} 
\end{equation}
Observe that if $N > 1$, the function $\phi_{N}$ is supported on
$\{ \xi : N / 2 \leq | \xi | \leq N (1 + 2^{- 100})\}$ and
$\phi_N(\xi) = 1$ when $(1 + 2^{- 100}) N / 2 < | \xi | < N$.  We define
\begin{equation}\label{eq:LP-projection}
\FT{\LP_{N} f} (\xi) = \phi_N \big(| \xi |\big) \FT{f} (\xi) .
\end{equation}
Note that this projection is different from the one introduced in
\eqref{eq:proj-unit-scale}. Next, we recall the classical Bernstein estimates.

\begin{lemma}\label{lem:bernstein}
For any $1 \leq r_1 \leq r_2 \leq \infty$ it holds that
\[
\| f \|_{L_x^{r_2} (\R^d)} \lesssim \diam \big( \spt \FT{f} \big)^{d \big( \frac{1}{r_1} - \frac{1}{r_2} \big)} \|f\|_{L_x^{r_1} (\R^d)} .
\]
In particular, since $\diam \big( \spt \Fourier(\QP_{k} f) \big) \approx 1$, \eqref{eq:unit-scale-bernstein} holds. We remark that  for $\LP_{N}$
  as in \eqref{eq:LP-projection}, $\diam \big( \spt \FT{\LP_{N} f} \big) \approx N$.
\end{lemma}

Next, we recall the Strichartz estimates for general self-adjoint operator
$\Lap$ of order $s$ with constant coefficients. We say that a pair $(p, q)$ is
$\Lap$-admissible if
\begin{equation} \label{eq:defn:admissible}
  \frac{\sigma}{p} + \frac{d}{q} = \frac{d}{2} \qquad   \text{and} \quad q \in \big[ 2, \frac{2 d}{d - \sigma} \big) .
\end{equation}

\begin{lemma}[{\cite{choRemarksDispersiveEstimates2011,dinhWellposednessRegularityIllposedness2018,ginibreSmoothingPropertiesRetarded1992,keelEndpointStrichartzEstimates1998}}] \label{lem:strichartz}
Fix $d > \sigma \geq 2$ and let $\Lap$ satisfy \eqref{eq:symbol-bddness} and
\eqref{eq:symbol-curvature}. Then, there exists $T_0 > 0$ such that
for any open interval $I \subset \R$ with $|I| \leq T_0$, we have
\begin{equation}\label{eq:strichartze1}
\big\| e^{- i t \Lap} f \big\|_{L_t^p L_x^q (I \times \R^d)} \lesssim \|f\|_{L^2 (\R^d)} . 
\end{equation}
\end{lemma}

\begin{remark}
  Since the evolution operator $e^{- i t \Lap}$ commutes with projections
  $\LP_{N}$ for any $N \in 2^{\Z}$ and with $\QP_{n}$ for any $n \in \Z^d$ (all 
  are Fourier multipliers), \eqref{eq:strichartze1}
  holds with $f$ replaced by $\LP_{N} f$ or by $\QP_{n} f$ on both
  sides of the inequality.
\end{remark}

In the rest of the paper, we assume that all time intervals have length less
than $T_0$ given in \Cref{lem:strichartz}, and therefore the Strichartz
estimates hold. Since we are only interested in the local existence, this
assumption does not influence our main results.

To introduce the directional norms, decompose $x \in \R^d$ for any $l \in \{1, \ldots, d\}$ as
\[ x = x_l e_l + \sum_{i = 1, i \neq l}^d x_i e_i = : x_l e_l + x'_l \]
and, if there is no possible confusion, we write $x' \assign x_l'$. Fix $I
\subset \R$ and $l \in \{1, \ldots, d\}$, and for $1 \leq a, b < \infty$
define
\[
\|h\|_{L_{e_l}^{a, b} (I \times \R^d)} = \Big( \int_{\R} \Big( \int_I \int_{\R^{d - 1}} |h (t, x_l e_l + x_l') |^b \dd x_l' \dd t \Big)^{\frac{a}{b}} \dd x_l \Big)^{\frac{1}{a}} \,,
\]
where $h : I \times \R^d \to \mathbb{C}$ is such that the right-hand side is
finite. When $a = \infty$ or $b = \infty$, we use the standard
modifications by the supremum norm. Next, we establish a maximal and a local
smoothing estimates for the directional norms.

\begin{lemma}\label{lem:dir-maximal}
Fix $d > \sigma \geq 2$ and $\Lap$ satisfying \eqref{eq:symbol-bddness} and
\eqref{eq:symbol-curvature}. There exists $T_0 > 0$ such that for any
open interval $I \subset \R$ with $|I| \leq T_0$, any
$l \in \{ 1, \ldots, d \}$, $N \in 2^{\N}$, and $f \in L^2 (\R^d)$ we have
\begin{equation}\label{eq:dir-maximal}
N^{- \frac{d - 1}{2}} \big\| e^{- i t \Lap} \LP_{N} f \big\|_{L_{e_l}^{2, \infty} (I \times \R^d)} \lesssim \|\LP_{N} f\|_{L_x^2 (\R^d)} \,.
\end{equation}
\end{lemma}

\begin{lemma}\label{lem:dir-local-smoothing}
Fix $d > \sigma \geq 2$ and $\Lap$ satisfying \eqref{eq:symbol-bddness} and
\eqref{eq:symbol-curvature}. There exists $T_0 > 0$ such that for any
open interval $I \subset \R$ with $|I| \leq T_0$, any
$l \in \{ 1, \ldots, d \}$, $N \in 2^{\N}$, and $f \in L^2 (\R^d)$ we have
\begin{equation}\label{eq:dir-local-smoothing}
N^{\frac{\sigma - 1}{2}} \big\| e^{- i t \Lap} \LP_{N} \UP_{e_l}^{\Lap} f \big\|_{L_{e_l}^{\infty, 2} (I \times \R^d)} \lesssim \|\LP_{N} f\|_{L^2 (\R^d)} \,, 
\end{equation}
where $\UP_{e_l}^{\Lap}$ is a frequency projection operators given by
$\FT{\UP_{e_l}^{\Lap} f} (\xi) \eqd \1_{\mf{U}_{e_l}} (\xi) \FT{f} (\xi)$ with
\begin{equation}\label{eq:microlocal-projections}
\mf{U}_{e_l} \eqd\begin{ltae}
\Big\{ \xi \in \R^d \st | \nabla \Lap (\xi) \cdot e_l | > \frac{| \nabla \Lap   (\xi) |}{2 \sqrt{d}} \Big\}
\\
\quad\setminus \bigcup_{l' = 1}^{l - 1} \Big\{ \xi \in \R^d \st | \nabla \Lap (\xi) \cdot e_{l'} | > \frac{| \nabla \Lap (\xi) |}{2 \sqrt{d}} \Big\} .
\end{ltae}
\end{equation}
\end{lemma}

\begin{remark}
Notice that the scaling power $N^{- \frac{d - 1}{2}}$ in
\eqref{eq:dir-maximal} does not depend on $\Lap$ but only on the dimension
$d$, whereas $N^{\frac{\sigma - 1}{2}}$ in \eqref{eq:dir-local-smoothing}
only depends on the order of $\Lap$ but not on the dimension $d$.
\end{remark}

\begin{proof*}{Proof of \Cref{lem:dir-maximal}}
We use a $T T^{\ast}$ argument for the operator $T$ given by
\begin{equation}\label{eq:T-operator-max}
T f (t, x) \eqd \int_{\R^d} e^{2 \pi i \xi x - i t \Lap (\xi)} \chi_N (\xi) \FT{f} (\xi) d \xi, 
\end{equation}
where
\[
\chi_N (\xi) \eqd \sum_{\substack{ N' \in \N\\ | N' - N | \leq 2 }} \phi_{N'} (| \xi |)
\]
so that $\chi_N (\xi) \phi_N (| \xi |) = \phi_N (| \xi |)$. Bound
\eqref{eq:dir-maximal} can be rewritten as
\[
\|T f\|_{L_{e_l}^{2, \infty} (I \times \R^d)} \lesssim N^{\frac{d -1}{2}} \|f\|_{L^2_x (\R^d)},
\]
or equivalently by duality $\|T^{\ast} g\|_{L^2_x (\R^d)}^2 \lesssim N^{d - 1} \|g\|_{L_{e_l}^{2, 1} (I \times \R^d)}^2$. We claim that it suffices to show that
\begin{equation}\label{eq:TTstar-goal}
\|T T^{\ast} g\|_{L_{e_l}^{2, \infty} (I \times \R^d)} \lesssim N^{d - 1}  \|g\|_{L_{e_l}^{2, 1} (I \times \R^d)} . 
\end{equation}
Indeed, if \eqref{eq:TTstar-goal} holds, then
\[
\|T^{\ast} g\|_{L^2_x (\R^d)}^2 = \langle T^{\ast} g, T^{\ast} g \rangle
= \langle T T^{\ast} g, g \rangle \lesssim \|T T^{\ast} g\|_{L^{2, \infty}_{e_l}} \|g\|_{L_{e_l}^{2, 1}}
\lesssim N^{d - 1} \|g\|_{L_{e_l}^{2, 1} (I \times \R^d)}^2
\]
as desired. Direct computations show that
\begin{equation}\label{eq:TTstar-kernel}
\begin{ltae}
T T^{\ast} g (t, x) = \int_{\R \times \R^d} K_N (t - s, x - y) g (s, y)
\dd s \dd y,
\\
K_N (t, x) \eqd \frac{1}{(2 \pi)^{2 d}} \int_{\R^d} e^{2 \pi i \xi \cdot
  x - i t \Lap (\xi)} \chi_N^2 (\xi) \dd \xi,
\end{ltae} 
\end{equation}
reducing our proof, by Young convolution inequality, to
\[
\| K_N \|_{L^{1, \infty}_{e_l} (\R \times \R^d)} \lesssim N^{d - 1} .
\]
For simplicity, we henceforth suppose $e_l = e_1$. Since $\chi_N$ is
supported on a set of measure of order $N^d$, by interchanging the
integral and absolute value, we have
\begin{equation}\label{eq:trve}
| K_N (t, x) | \lesssim N^d \,.
\end{equation}
In addition, for $x_1 \neq 0$, an integration by parts, and an oscillation
of the linear phase yield  a decay in $x_1$:
\begin{equation}\label{eq:wims}
| K_N (t, x) | \begin{ltae}
= \frac{1}{4 \pi^2} \Big| \int_{\R^d}   \frac{1}{x_1^2}
\partial^2_{\xi_1} \Big( e^{2 \pi i x_1 \xi_1} \Big)  e^{2 \pi i \xi' \cdot x' - i t \Lap (\xi)} \chi_N^2 (\xi) \dd \xi \Big|
\\
\lesssim \frac{N^d}{| x_1 |^2} \big\| e^{- i t \Lap (\xi)} \chi_N^2 (\xi) \big\|_{C^2} .
\end{ltae} 
\end{equation}
Fix $N_0 \in 2^N$ large enough, depending on $\Lap$, so that
\eqref{eq:symbol-bddness} and \eqref{eq:symbol-curvature} hold for any
$\xi \in \spt (\chi_N)$ and $N > N_0$.  Since $\chi_N$ and $\Lap$ are smooth and
$\chi_N$ is compactly supported, we obtain \eqref{eq:trve} from
\eqref{eq:wims}, and therefore \eqref{eq:dir-maximal} holds for all
$N \leq N_0$.
  
Thus, it remains to consider $N > N_0$ and in particular
\eqref{eq:symbol-bddness} and \eqref{eq:symbol-curvature} hold. Then,
by {\cite[Lemma 2.1]{huangInhomogeneousOscillatoryIntegrals2017}} we
obtain the bound
\begin{equation}\label{eq:efou}
K_N (t, x_1, x') \lesssim |t|^{- \frac{d}{\sigma}} \,.
\end{equation}
Hence, if $|x_1 | \lesssim \langle N\rangle^{\sigma - 1} |t|$, then \eqref{eq:trve} and
\eqref{eq:efou} imply
\[
| K_N (t, x_1, x') | \lesssim \min \Big( N^d, |t|^{- \frac{d}{\sigma}}\Big) \lesssim \min \Big(
N^d, N^{d - \frac{d}{\sigma}} |x_1 |^{- \frac{d}{\sigma}} \Big) \lesssim N^d \big\langle Nx_1 \big\rangle^{-
  \frac{d}{\sigma}}
\]
and since $d > \sigma$
\begin{equation}\label{eq:flfe}
\Big\| \1_{|x_1 | \lesssim \langle N \rangle^{\sigma - 1} |t|} K_N (t, x_1, x') \Big\|_{L^{1, \infty}_{e_l} (\R \times \R^d)} \leq \int_{\R} N^d \big\langle Nx_1 \big\rangle^{- \frac{d}{\sigma}} \dd x_1
\lesssim N^{d - 1}
\end{equation}
and \eqref{eq:dir-maximal} follows.
Finally, we restrict to $| x_1 | \gtrsim N^{\sigma - 1} |t|$, where 
we obtain a lower bound on the derivative of the
phase
\[
\partial_{\xi_1} \Big( 2 \pi (x_1 \xi_1 + x' \cdot \xi') - t \Lap (\xi) \Big) = 2 \pi x_1 - t \partial_{\xi_1}  \Lap (\xi) \,.
\]
Indeed, by \eqref{eq:symbol-bddness}, for $\xi \in \spt(\chi_N)$ we have  $|
\partial_{\xi_1} \Lap (\xi) | \lesssim | \xi |^{\sigma - 1} \lesssim N^{\sigma - 1}$. In particular, for $|x_1| \gtrsim N^{\sigma - 1}|t|$ one has 
\begin{equation}\label{eq:lbest}
\big|2 \pi x_1 - t \partial_{\xi_1} \Lap (\xi) \big| \gtrsim | x_1 | .
\end{equation}
To show a decay of $K_N$ in $x_1$, we exploit oscillations of the phase. By
integrating by parts twice we obtain
\[
K_N (t, x_1, x') \begin{ltae}
= - \int_{\mathclap{\R \times \R^{d - 1}}} \chi_N^2 (\xi) \Big( \frac{1}{2 \pi x_1 - t  \partial_{\xi_1} \Lap (\xi)} \partial_{\xi_1} \Big)^2 e^{2 \pi i x \cdot \xi  - i t \Lap (\xi)}  d \xi_1 d \xi'
\\
= - \int_{\mathclap{\R \times \R^{d - 1}}} \begin{ltae}
e^{2 \pi i x \cdot \xi  - i t \Lap (\xi)}
\\
\quad\times \partial_{\xi_1} \Big( \frac{1}{2 \pi x_1 - t  \partial_{\xi_1} \Lap (\xi)} \partial_{\xi_1} \Big( \frac{\chi_N^2  (\xi)}{2 \pi x_1 - t \partial_{\xi_1} \Lap (\xi)} \Big) \Big) d \xi_1 d \xi' .
\end{ltae}
\end{ltae} \]
By rescaling,
\[
| \partial_{\xi_1}^k \chi_N^2 (\xi) | \lesssim N^{- k} \1_{| \xi |  \approx N}
\]
while by \eqref{eq:symbol-bddness} and \eqref{eq:lbest}
\[
\Big| \partial_{\xi_1} \big( \frac{1}{2 \pi x_1 - t \partial_{\xi_1}\Lap (\xi)} \big) \Big|
\begin{ltae}
= \frac{\big| t \partial_{\xi_1, \xi_1}^2 \Lap (\xi) \big|}{\big| 2 \pi x_1 - t \partial_{\xi_1} \Lap (\xi) \big|^2}
\\
\lesssim \frac{|t| \,| \xi |^{\sigma - 2}}{|x_1 |^2}  \lesssim \frac{|t| N^{\sigma - 2}}{|x_1 |^2} \lesssim \frac{1}{N |x_1 |}
\end{ltae}
\]
and similarly
\[
\Big| \partial^2_{\xi_1, \xi_1} \Big( \frac{1}{2 \pi x_1 - t \partial_{\xi_1} \Lap (\xi)} \Big) \Big| \lesssim \frac{1}{N^2 |x_1 |} .
\]
Combining these estimates we obtain
\[
\Big| \partial_{\xi_1} \Big( \frac{1}{2 \pi x_1 - t \partial_{\xi_1} \Lap (\xi)} \partial_{\xi_1} \big( \frac{\chi_N^2 (\xi)}{2 \pi x_1 - t \partial_{\xi_1} \Lap (\xi)} \big) \Big) \Big|
\lesssim  \frac{1}{N^2 | x_1 |^2} \1_{| \xi |
  \approx N},
\]
and after integration over in $\xi$,
\[
\1_{| x_1 | \gtrsim t N^{\sigma - 1}} | K_N (t, x_1, x') | \lesssim N^{d - 2} | x_1 |^{- 2} .
\]
Since by \eqref{eq:trve} one has  
$| K_N (t, x_1, x') | < N^d$, we obtain
\[
\1_{| x_1 | \gtrsim t N^{\sigma - 1}} | K_N (t, x_1, x') | \lesssim N^d \langle N x_1 \rangle^{- 2} \,,
\]
and consequently
\[
\big\| \1_{| x_1 | \gtrsim t N^{\sigma - 1}} K_N (t, x_1, x') \big\|_{L^{1, \infty}_{e_1} (\R \times \R^d)} \leq \int_{\R} N^d  \langle N x_1 \rangle^{- 2} \dd x_1 \lesssim N^{d - 1},
\]
as required.
\end{proof*}

\begin{proof*}{Proof of \Cref{lem:dir-local-smoothing}}
Since the $\UP_{e_l}^{\Lap}$ is a bounded projection on $L^2$, it commutes
with the evolution $e^{i t \Lap}$ and with the Littlewood-Paley projections $\LP_{N}$.
Without loss of generality assume $e_l = e_1$ and 
let $T$ be defined
analogously to \eqref{eq:T-operator-max} as
\[
T f (t, x) \eqd \int_{\R^d} e^{2 \pi i \xi \cdot x - i t \Lap (\xi)} \chi_N (\xi) \1_{\mf{U}_{e_1}} (\xi) \FT{f} (\xi) d \xi \,.
\]
The assertion of the lemma is equivalent to
\[
\sup_{x_1 \in \R} \big\| T f (t, x_1, x') \big\|_{L^2_t L^2_{x'} ( I \times \R^{d - 1} )} \lesssim N^{-\frac{\sigma - 1}{2}} \| f \|_{L^2 ( \R^d )}.
\]
Using the Plancherel's identity in $x'$ gives that 
\[
\begin{ltae}
\big\| T f (t, x_1, x') \big\|_{L^2_t L^2_{x'} ( I \times \R^{d - 1} )} 
\\
\quad =
\int_{I\times\R^{d-1}} \Big| \int_{\R} e^{i 2 \pi \xi_1 x_1  - i t \Lap (\xi)} \chi_N (\xi) \1_{\mf{U}_{e_1}} (\xi) \FT{f} (\xi_1,\xi') \dd \xi_1 \Big|^2 \dd t \dd\xi'    
\end{ltae}
\]
Thus, let us fix $x_1 \in \R$ and $\xi' \in \R^{d- 1}$ and show that 
\begin{equation}\label{eq:ibfls}
\int_I \Big| \int_{\R} e^{i 2 \pi \xi_1 x_1  - i t  \Lap (\xi)} \chi_N (\xi) \1_{\mf{U}_{e_1}} (\xi) \FT{f} (\xi) d\xi_1 \Big|^2 \dd t
\\
\lesssim N^{-(\sigma -1)}\int_{\R} | \FT{f} (\xi) |^2 \dd   \xi_1 \,,
\end{equation}
where, as usual, $\xi = (\xi_1, \xi')$. Choose $N_0 \in 2^{\N}$ large enough that \eqref{eq:symbol-bddness} and 
\eqref{eq:symbol-curvature} hold for any $(\xi_1^0, \xi')$ with $|\xi_1^0, \xi'| \geq N$.

If $N \leq N_0$, by Cauchy-Schwarz inequality and $| I | \lesssim T_0$ one has
that the left-hand side of \eqref{eq:ibfls} is bounded by
$N | I | \big\| \FT{f} (\xi_1, \xi')\big\|_{L^2_{\xi_1}}^2$ and our claim follows.

Assume $N > N_0$. 
Since $\xi_1 \rightarrow \Lap (\xi_1, \xi')$ is smooth, the set $\big\{
\xi_1 \in \Omega_{\xi'} \suchthat | \nabla \Lap (\xi_1, \xi') \cdot e_1| \neq 0 \big\} \subset \R$ is open, and it can be represented as a countable union
of disjoint open intervals:
\[
\big\{ \xi_1 \in \Omega_{\xi'} \suchthat | \nabla \Lap (\xi_1,  \xi') \cdot e_1| \neq 0 \big\} = \bigcup_{\mc{W} \in \mc{I}}
\mc{W} .
\]
The integrand on the left-hand side of \eqref{eq:ibfls} vanishes unless
$(\xi_1, \xi') \in \mf{U}_{e_l}$ or unless
\[
\Big| \nabla \Lap (\xi_1, \xi') \cdot e_1 \Big| > \frac{\big| \nabla \Lap (\xi_1, \xi') \big|}{2 \sqrt{d}} .
\]

Let $\mc{I}_0 \subset \mc{I}$ be the sub-collection of intervals such that $\mc{W} \in \mc{I}_0$ if and only if
$\mc{W} \cap \{ \xi_1 \st (\xi_1, \xi') \in \mf{U}_{e_1}\} \neq \emptyset$. We claim that $\mc{I}_0$ is finite with cardinality independent of $N$. 
Indeed, for fixed $(\xi_1^-, \xi_1^+) := \mc{W} \in \mc{I}_0$ we can without loss of generality assume $\mc{W} \subset \textrm{supp} (\chi_N(\cdot, \xi') )$, because
there are at most four intervals in $\mc{I}$ that contain an endpoint (one of at most four) of  $\spt (\chi_N(\cdot, \xi') )$.  
Fix $\xi_1^0 \in \mc{W} \cap \mf{U}_{e_1}$, and in particular 
\begin{equation}\label{eq:lbfgt}
\big| \nabla \Lap (\xi_1^0, \xi') \cdot e_1 \big| > \frac{\big|\nabla \Lap (\xi_1^0, \xi') \big|}{2 \sqrt{d}} .
\end{equation}
Since $(\xi_1^0, \xi') \in \spt (\chi_N)$, then
$|(\xi_1^0, \xi')| \gtrsim N \geq N_0$, and
$\big| \nabla \Lap (\xi_1^0, \xi') \big| \gtrsim N^{\sigma - 1}$ by \eqref{eq:symbol-bddness}.
In particular,
$| \nabla \Lap (\xi_1^0, \xi') \cdot e_1 | \gtrsim N^{\sigma - 1}$ by \eqref{eq:lbfgt}.  On
the other hand,
$\nabla \Lap (\xi_1^-, \xi') \cdot e_1 = \nabla \Lap (\xi_1^+, \xi') \cdot e_1 = 0$. The bound
on the second derivative in \eqref{eq:symbol-bddness} yields
$| \partial_{\xi_1} ( \nabla \Lap (\xi_1^-, \xi') \cdot e_1 ) | = | \partial_{\xi_1}^2 \Lap (\xi_1^-,
\xi')| \lesssim N^{\sigma - 2}$.  Then, by the mean value theorem,
$| \mc{W}| \gtrsim N$. Since the intervals $\mc{W} \in \mc{I}_0$ are pairwise
disjoint and $\mc{W} \subset \spt (\chi_N(\cdot, \xi') )$, the claimed bound
on the cardinality of $\mc{I}_0$ follows.

Observe that $\xi_1 \mapsto \partial_{\xi_l} \Lap (\xi_1, \xi')$ does not change sign on
$\mc{W}$. Hence, in \eqref{eq:ibfls} restricted to $\mc{W}$, we use the
change of variables $\xi_1 \rightarrow \theta = \Lap (\xi_1, \xi')$, apply Plancherel
identity in the variable $t$, and change back to the variable $\xi_1$
and obtain after a use of \eqref{eq:lbest} on $\mf{U}_{e_1}$ and
\eqref{eq:symbol-bddness} on $\spt (\chi_N(\cdot, \xi') )$


\[ \begin{ltae} \int_I \Big| \int_{\mathclap{\mc{W}}} e^{2 \pi i \xi_1 x_1 - i t \Lap
  (\xi_1, \xi')} \chi_N (\xi) \1_{\mf{U}_{e_1}} (\xi) \FT{f} (\xi_1, \xi') d \xi_1 \Big|^2 \dd t
\\
\leq \int_{\R} \Big| \int_{\mathclap{\theta ( \mc{W} )}} e^{2 \pi i x_1 \xi_1 (\theta) - i t \theta}
\chi_N (\xi_1 (\theta), \xi') \1_{\mf{U}_{e_1}} (\xi_1 (\theta), \xi') \FT{f} (\xi_1(\theta), \xi') \frac{d
  \theta}{\partial_{\xi_1} \Lap (\xi_1(\theta), \xi')} \Big|^{2} \dd t
\\
\leq \int_{\mathclap{\theta ( \mc{W} )}} \Big| e^{2 \pi i (x_1 \xi_1 (\theta) + \xi' \cdot x')} \chi_N
(\xi_1 (\theta), \xi') \1_{\mf{U}_{e_1}} (\xi_1 (\theta), \xi')
\FT{f} (\xi_1 (\theta), \xi') \frac{1}{\partial_{\xi_1} \Lap (\xi_1 (\theta), \xi')} \Big|^2 d \theta\\
\leq \int_{\R} \Big| \chi_N(\xi_1, \xi') \1_{\mf{U}_{e_l}} (\xi_1, \xi') \FT{f}(\xi_1, \xi')
\frac{1}{\partial_{\xi_1} \Lap (\xi_1, \xi')} \Big|^2 \big|\partial_{\xi_1} \Lap (\xi_1, \xi') \big| d \xi_1
\\
\lesssim \int_{\R} \chi_N(\xi_1, \xi') | \FT{f} (\xi_1, \xi') |^2 \frac{d \xi_1}{\big | \nabla \Lap (\xi_1,
  \xi') \big|} \lesssim N^{- (\sigma -1)} \int_{\R} | \FT{f} (\xi_1, \xi') |^2 d \xi_1,
\end{ltae} \]
as desired.
\end{proof*}

\section{Linear non-homogeneous estimates}\label{sec:linear-nonhomogeneous}

This section contains estimates for the solution $v$ of the non-homogeneous
Schrö\-din\-ger equation
\begin{equation}\label{eq:liheq}
\begin{cases}
(i \partial_t - \Lap) v = h & \text{ on } I \times \R^d,
\\
v (t_0) = 0 & 
\end{cases} 
\end{equation}
in Besov -type spaces. Specifically, 
for $\epsilon \in [0, 1)$ and $\xs > 0$ we set
\begin{equation}\label{eq:defn:X-norm} 
\begin{ltae}
\|v\|_{X^{\xs, \epsilon} (I)} \eqd \Big( \sum_{N \in 2^{\N}} N^{2 \xs}  \|\LP_{N} v\|^2_{X^{\epsilon}_N (I)} \Big)^{\frac{1}{2}},\\
\|v\|_{X^{\epsilon}_N (I)} \eqd \begin{ltae} \| v \|_{L_t^{\infty} L_x^{\frac{2}{1 - \epsilon}} (I \times
  \R^d)} + \| v \|_{L_t^{2 \frac{d + \sigma}{d}} L_x^{2 \frac{d + \sigma}{d}} (I \times \R^d)}
\\
+ \sum_{l = 1}^d N^{- \frac{d - 1}{2}} \| v \|_{L_{e_l}^{\frac{2}{1 -\epsilon}, \infty} (I \times
  \R^d)}
\\
+ \sum_{l = 1}^d N^{\frac{\sigma - 1}{2}} \Big\| \UP_{e_l}^{\Lap} v \Big\|_{L_{e_l}^{\infty,
    \frac{2}{1 - \epsilon}} (I \times \R^d)} \,,
\end{ltae}
\end{ltae} 
\end{equation}
where $\UP_{e_l}^{\Lap}$ is as in Lemma \ref{lem:dir-local-smoothing}.
Observe that the norm of $X^{\xs, \epsilon}$   
contains Strichartz-type component $L_t^{2 \frac{d + \sigma}{d}} L_x^{2 \frac{d + \sigma}{d}} (I\times \R^d)$, which is admissible, 
and $L_t^{\infty} L_x^{\frac{2}{1 - \epsilon}} (I \times \R^d)$ which is close to an admissible space 
$L_t^{\infty} L_x^{2} (I \times \R^d)$ if $\epsilon > 0$ is small, see \Cref{lem:strichartz}. 
The other two components are close to  
directional maximal (see \Cref{lem:dir-maximal}), and
directional local smoothing (see \Cref{lem:dir-local-smoothing}) spaces.

We assume that the non-homogeneous term $h$ is controlled in a space
dual to $X^{\xs,\epsilon} (I)$. We set
\begin{equation}\label{eq:defn:X-norm-dual}
\begin{ltae}
\| h \|_{X^{\xs, \epsilon} (I)^{\ast}} \eqd \Big( \sum_{N \in 2^{\N}} N^{2 \xs} \|\LP_{N} h\|^2_{X^{\epsilon}_N (I)^{\ast}} \Big)^{\frac{1}{2}}
\\
\|h\|_{X^{\epsilon}_N (I)^{\ast}}  \eqd \sup \Big\{ \Big| \int_{I \times \R^d} h (x, t) v_{\ast} (x, t) \dd t \dd x \Big| \st \| v_{\ast} \|_{X^{\epsilon}_N (I)} \leq 1 \Big\} .
\end{ltae} 
\end{equation}
The main result of this section is the following a priori bound on $v$ in terms of $h$.

\begin{proposition}
\label{prop:main-linear-estimate} The solution $v : I \times \R^d
\rightarrow \mathbb{C}$ to
\[ \begin{cases}
(i \partial_t - \Lap) v = h & \text{ on } I \times \R^d,\\
v (0) = 0 &
\end{cases} \]
satisfies for any $\epsilon \in\big( 0, \frac{\sigma}{d + \sigma}\big)$ and any $N \in 2^{\N}$ the bound
\[ {\| \LP_{N} v \|_{X^{\epsilon}_N (I)}}  \lesssim_{\epsilon} N^{| O(\epsilon) |} \| \LP_{N} h \|_{X^{\epsilon}_N (I)^{\ast}} \,, \]
where  the function $\epsilon \mapsto | O (\epsilon) |$ depends only on $\Lap,
d, \sigma$, and $\xs > 0$.
As a consequence,
\[ \|v\|_{X^{\xs, \epsilon} (I)} \lesssim \| h \|_{X^{\xs + | O(\epsilon) |, \epsilon} (I)^{\ast}} \]
for any $\xs > 0$.
\end{proposition}

The proof Proposition \ref{prop:main-linear-estimate} 
relies on a sequence of lemmata that combine interpolation, dual
versions of \Cref{lem:strichartz}, \Cref{lem:dir-maximal}, and
\Cref{lem:dir-local-smoothing} and the ``Christ-Kiselev'' lemma \cite{christMaximalFunctionsAssociated2001}. The solution to \eqref{eq:liheq} is given by
the Duhamel formula:
\begin{equation} \label{eq:duhamel}
v (t, x) = \int_0^t e^{- i (t - s) \Lap} h (s, x) \dd s.
\end{equation}
The results of this section can be seen as properties of the mapping $h
\mapsto v$ between spaces $X^{\xs + | O (\epsilon) |,
  \epsilon} (I)^{\ast}$  and $X^{\xs, \epsilon} (I)$.

\begin{lemma} \label{lem:linear-evolution-disjoint-times}
Fix a time interval $I \subset \R$ with $0 \in I, | I | < T_0$ and let 
$J, J' \subset I$ be two 
disjoint sub-intervals $I$. For any $N \in 2^{\N}$ it holds that
\[ \Big\| \int_{J'} e^{- i (t - s) \Lap} \LP_{N} h (s) \dd s \Big\|_{X^{\epsilon}_N (J)} \lesssim_{\epsilon} N^{| O (\epsilon) |} \| \LP_{N} h \|_{X^{\epsilon}_N (J')^{\ast}} \,.
\]
We stress hat the implicit constant do not depend on $J$, $J'$, or $N$.
\end{lemma}

\begin{lemma}\label{lem:linear-evolution-christ-kiselev}
Let $I \subset \R$ be a time interval with $0 \in I, | I | < T_0$. For any
$\epsilon \in \Big( 0,\frac{\sigma}{d} \Big)$ it holds that
\[ \Big\| \int_0^t e^{- i (t - s) \Lap} \LP_{N} h (s) \dd s \Big\|_{X^{\epsilon}_N (I)} \lesssim_{\epsilon} N^{| O (\epsilon) |} \| \LP_{N} h (s) \|_{X^{\epsilon}_N (I)^{\ast}} \]
for any $N \in 2^{\N}$. The implicit constants do not depend on $N$.
\end{lemma}

\begin{proof*}{Proof of \Cref{lem:linear-evolution-disjoint-times}}
Note that \Cref{lem:strichartz}, Bernstein inequality (Lemma
\ref{lem:bernstein}), and Fubini theorem imply
\[
\big\| e^{- i s \Lap} \LP_{N} u \big\|_{L_{e_l}^{\infty, \infty} ( I\times \R^d )}
\begin{ltae}
= \big\| e^{- i s \Lap} \LP_{N} u \big\|_{L^{\infty}_t L^{\infty}_x  ( I \times \R^d )}
\\
\lesssim N^{\frac{d}{2}} \big\| e^{- i s \Lap} \LP_{N} u \big\|_{L^{\infty}_t L^2_x ( I \times \R^d )} \lesssim  N^{\frac{d}{2}} \| \LP_{N} u \|_{L^2 ( \R^d )}\,.
\end{ltae} \] Then, interpolating \Cref{lem:strichartz},
\Cref{lem:dir-maximal}, and \Cref{lem:dir-local-smoothing} and using
the Hölder inequality implies
\begin{align*} 
\big\| e^{- i s \Lap} \LP_{N} u \big\|_{L^{\infty}_t L^{\frac{2}{1 -   \epsilon}}_x ( I \times \R^d )}
&\lesssim \big\| e^{- i s \Lap} \LP_{N} u \big\|_{L^{\infty}_t L^2_x ( I \times \R^d )}^{1 - \epsilon} \big\| e^{- i s \Lap} \LP_{N} u \big\|_{L^{\infty}_t L^{\infty}_x ( I \times \R^d)}^{\epsilon}
\\
&\lesssim N^{ | O (\epsilon) |} \| \LP_{N} u \|_{L^2( \R^d )} \,,
\end{align*}
and 
\[\big\| e^{- i s \Lap} \LP_{N} u \big\|_{L_t^{2 \frac{d + \sigma}{d}} L_x^{2 \frac{d + \sigma}{d}} (I \times \R^d)}  \lesssim  \| \LP_{N} u \|_{L^2 ( \R^d )} \,,
\]
and
\begin{align*}
\big\| e^{- i s \Lap} \LP_{N} u \big\|_{L_{e_l}^{\frac{2}{1 - \epsilon},\infty} ( I \times \R^d )}
&\lesssim \big\| e^{- i s \Lap} \LP_{N} u \big\|_{L_{e_l}^{2, \infty} ( I \times \R^d )}^{1 - \epsilon} \big\| e^{- i s \Lap} \LP_{N} u \big\|_{L_{e_l}^{\infty, \infty} ( I \times \R^d)}^{\epsilon}
\\ 
&\lesssim N^{\frac{d - 1}{2} + | O (\epsilon) |} \| \LP_{N} u \|_{L^2( \R^d )},
\end{align*}
and
\begin{align*}
\big\| e^{- i s \Lap} \LP_{N} \UP_{e_l}^{\Lap} u \big\|_{L_{e_l}^{\infty,  \frac{2}{1 - \epsilon}} ( I \times \R^d )}
 &\lesssim
 \begin{ltae}
 \big\| e^{- i s \Lap} \LP_{N} \UP_{e_l}^{\Lap} u   \big\|_{L_{e_l}^{\infty, 2} ( I \times \R^d )}^{1 -  \epsilon}
\\ \quad\times\big\| e^{- i s \Lap} \LP_{N} \UP_{e_l}^{\Lap} u  \big\|_{L_{e_l}^{\infty, \infty} ( I \times \R^d)}^{\epsilon}
 \end{ltae}
 \\
&\lesssim N^{- \frac{\sigma - 1}{2} + | O (\epsilon) |} \big\| \LP_{N} \UP_{e_l}^{\Lap} u \big\|_{L^2 ( \R^d )} .
\end{align*}
In summary, for $I = J$
\begin{equation*}
\big\| e^{- i s \Lap} \LP_{N} u \big\|_{X^{\epsilon}_N (J)} \lesssim N^{ | O (\epsilon) |} \| \LP_{N} u \|_{L^2( \R^d )}
\end{equation*}
and by duality for $I = J'$
\begin{equation*}
\Big\| \int_{J'} e^{ i s \Lap} \LP_{N} v ds \Big\|_{L^2(\R^d)} \lesssim  N^{ | O (\epsilon) |} \| \LP_{N} v \|_{X^{\epsilon}_N (J')*} \,.
\end{equation*}
Consequently,
\[
\Big\| \int_{J'} e^{- i (t - s) \Lap} \LP_{N} h (s) \dd s \Big\|_{X^{\epsilon}_N (J)}
\begin{ltae}
\lesssim N^{ | O (\epsilon) |} \Big\| \int_{J'} e^{ i s \Lap} \LP_{N} h (s) \dd s \Big\|_{L^2(\R^d)}
\\
\lesssim N^{ | O (\epsilon) |} \| \LP_{N} h \|_{X^{\epsilon}_N (J')*} \,,
\end{ltae}
\]
as desired.
\end{proof*}

\begin{proof*}{Proof of \Cref{lem:linear-evolution-christ-kiselev}}
 Let $v_1, v_2 \st I \times \R^d \rightarrow \C$ are 
 supported on disjoint time intervals $J_1, J_2 \subset I$, then we claim that there are norms
 $\tilde{X}^{\epsilon}_N (I)$ equivalent to $X^{\epsilon}_N (I)$ with
 some constant independent of $N$, such that
  \begin{equation}\label{eq:X-concat-sim}
  \big\| v_1 + v_2 \big\|_{\tilde{X}^{\epsilon}_N (I)}^{\frac{2}{1 -  \epsilon}} \leq \| v_1 \|_{\tilde{X}^{\epsilon}_N (J_1)}^{\frac{2}{1- \epsilon}} + \| v_2 \|_{\tilde{X}^{\epsilon}_N (J_2)}^{\frac{2}{1- \epsilon}}
    \end{equation}
  and if $h_1, h_2 \st I \times \R^d \rightarrow \C$ are 
 supported on disjoint time intervals $J_1, J_2
 \subset I$, then
  \begin{equation}\label{eq:Xstar-concat-sim}
    \big\| h_1 + h_2 \big\|_{\tilde{X}^{\epsilon}_N (I)^{\ast}}^{\frac{2}{1 +
    \epsilon}} \geq \| h_1 \|_{\tilde{X}^{\epsilon}_N (J_1)^{\ast}}^{\frac{2}{1 + \epsilon}} + \| h_2 \|_{\tilde{X}^{\epsilon}_N (J_2)^{\ast}}^{\frac{2}{1 + \epsilon}}\,,
      \end{equation}
  where the norm $\tilde{X}^{\epsilon}_N (I)^{\ast}$ is the dual norm to
  $\tilde{X}^{\epsilon}_N (I)$ (cf \eqref{eq:defn:X-norm-dual}). 
Consequently, by induction
\begin{equation}\label{eq:X-concat}
\Big\| \sum_i v_i  \Big\|_{\tilde{X}^{\epsilon}_N (I)}^{\frac{2}{1 -\epsilon}}
\leq
\sum_i \| v_i \|_{\tilde{X}^{\epsilon}_N (J_1)}^{\frac{2}{1 - \epsilon}} 
\end{equation}
and
\begin{equation}\label{eq:Xstar-concat}
\Big\| \sum_i h_i \Big\|_{\tilde{X}^{\epsilon}_N (I)^{\ast}}^{\frac{2}{1 + \epsilon}} \geq \sum_i  \| h_i \|_{\tilde{X}^{\epsilon}_N (J_1)^{\ast}}^{\frac{2}{1 + \epsilon}} \,.
\end{equation}

Fix $h \st I \times \R^d \rightarrow \C$ and without loss of generality suppose
$I = [0, T_0]$. Then, there is a sequence of interval
$\big\{ I^n_k = [t^n_k, t^n_{k + 1}) \big\}_{n \in \N, k \in \{ 0, \ldots, 2^n - 1 \}}$
such that $t_k^n \leq t^n_{k + 1}$, the intervals
$\{ I^n_k \}_{k \in \{ 0, \ldots, 2^n - 1 \}}$ partition $I$ and
$I^n_k = I_{2 k}^{n + 1} \cup I_{2 k + 1}^{n + 1}$. Furthermore, the
intervals are constructed so that
\begin{equation}\label{eq:dgct}   
\| h  \|_{\tilde{X}^{\epsilon}_N (I^n_k)^{\ast}} \lesssim 2^{- \frac{1 + \epsilon}{2}n} \| h \|_{\tilde{X}^{\epsilon}_N (I)^{\ast}} . 
\end{equation}
We postpone the construction of such sequence till the end of the proof.

Using the constructed intervals,  we have for any $t \in I$
\[
\1_{[0, t)} (s) = \sum_{n = 1}^{\infty} {\sum^{2^{n }-1}_{k = 0}}  \1_{I_{ k}^n} (s) \1_{I_{ k + 1}^n} (t) 
\]
and by triangle inequality and \eqref{eq:X-concat} for any regular $(t, s, x) \mapsto F(t, s, x)$
\[ \begin{ltae}
\Big\| \int_0^t F(t, s) \dd s \Big\|_{\tilde{X}^{\epsilon}_N (I)}  = \Big\| \int_I \1_{[0, t]} (s) F(t, s) \dd s \Big\|_{\tilde{X}^{\epsilon}_N (I)}
\\ \qquad \begin{ltae}
= \Big\| \int_I \sum_{n = 1}^{\infty} \sum^{2^{n}-1}_{k = 0}  \1_{I_{ k}^n} (s) \1_{I_{ k + 1}^n} (t) F(t, s) \dd s \Big\|_{\tilde{X}^{\epsilon}_N (I)}
\\
\leq \sum_{n = 1}^{\infty} \Big\| \int_I \sum^{2^{n} - 1}_{k = 0} \1_{I_{ k}^n} (s) \1_{I_{ k + 1}^n} (t) F(t, s) \dd s \Big\|_{\tilde{X}^{\epsilon}_N (I)}
\\
\leq \sum_{n = 1}^{\infty} \Big( \sum^{2^{n} - 1}_{k = 0}   \Big\| \int_I \1_{I_{ k}^n} (s) \1_{I_{ k + 1}^n} (t) F(t, s) \dd s \Big\|_{\tilde{X}^{\epsilon}_N (I)}^{\frac{2}{1 - \epsilon}} \Big)^{\frac{1 - \epsilon}{2}}
\\
\lesssim \sum_{n = 1}^{\infty} \Big( {\sum^{2^{n} - 1}_{k = 0}}  \Big\| \int_{I_{ k}^n} F(t, s) \dd s \Big\|_{\tilde{X}^{\epsilon}_N (I_{ k + 1}^n)}^{\frac{2}{1 - \epsilon}} \Big)^{\frac{1 - \epsilon}{2}} \,,
\end{ltae}
\end{ltae}
\]
where we suppressed the dependence of $F$ on $x$. 
Then, since norms $X^{\epsilon}_N (I)$ and $\tilde{X}^{\epsilon}_N (I)$ are equivalent and 
$I_{ k}^n$ and $I_{ k + 1}^n$ are disjoint, it follows from  \Cref{lem:linear-evolution-disjoint-times} with 
$F(t, s) = e^{- i (t - s) \Lap} \LP_{N} h (s)$ and \eqref{eq:dgct} that
\[
\begin{ltae}
\Big\| \int_0^t e^{- i (t - s) \Lap} \LP_{N} h (s) \dd s  \Big\|_{X^{\epsilon}_N (I)}
\\\quad \begin{ltae}
\lesssim 
\sum_{n = 1}^{\infty} \Big( {\sum^{2^{n} - 1}_{k = 0}}  \Big\| \int_{I_{ k}^n} e^{- i (t - s) \Lap} \LP_{N} h (s) \dd s \Big\|_{X^{\epsilon}_N (I_{ k + 1}^n)}^{\frac{2}{1 - \epsilon}} \Big)^{\frac{1 - \epsilon}{2}}\\
\lesssim N^{|O(\epsilon)|} \sum_{n = 1}^{\infty} \Big( \sum^{2^{n} - 1}_{k = 0}( {2^{- \frac{1 + \epsilon}{2} n}}  )^{\frac{2}{1 - \epsilon}} \Big)^{\frac{1 - \epsilon}{2}} \| \LP_{N} h
\|_{X^{\epsilon}_N (I)^{\ast}}
\\
\lesssim N^{|O(\epsilon)|} \sum_{n = 1}^{\infty}  \big( {2^n}  2^{- \frac{1 + \epsilon}{1 - \epsilon} n} \big)^{\frac{1 - \epsilon}{2}} \| \LP_{N} h \|_{X^{\epsilon}_N (I)^{\ast}}
\\
\lesssim \| \LP_{N} h \|_{X^{\epsilon}_N (I)^{\ast}} \,.
\end{ltae}
\end{ltae}
\]
Next, let us prove  \eqref{eq:X-concat-sim}, and \eqref{eq:Xstar-concat-sim}. We claim that 
\eqref{eq:Xstar-concat-sim} follows from \eqref{eq:X-concat-sim} by duality. 
Indeed, fix $h_j$ supported on $J_j$, $j \in \{1, 2\}$, with $J_1 \cap J_2 = \emptyset$. Then, 
\[ \| h_1 + h_2 \|_{\tilde{X}^{\epsilon}_N (I)^{\ast}} \begin{ltae}
\geq \sup_{v_1, v_2} \Big( \frac{\int_{J_1 \times \R^d} h_1 v_1 + \int_{J_2 \times \R^d} h_2 v_2}{\| v_1 + v_2 \|_{\tilde{X}^{\epsilon}_N (I)}} \Big)
\\
\geq \sup_{v_1, v_2} \bigg( \frac{\int_{J_1 \times \R^d} h_1 v_1 + \int_{J_2 \times \R^d} h_2 v_2}{\big( \| v_1 \|_{\tilde{X}^{\epsilon}_N (J_1)}^{\frac{2}{1 - \epsilon}} + \| v_2 \|_{\tilde{X}^{\epsilon}_N (I)}^{\frac{2}{1 - \epsilon}} \big)^{\frac{1 - \epsilon}{2}}} \bigg)
\end{ltae} \]
where the supremum is taken over $v_1, v_2 \in \tilde{X}^{\epsilon}_N
(I)$ supported on $J_1$ and $J_2$ respectively. 
We choose $v_j$ which maximize
$ v \mapsto \Big|\int_{J_j \times \R^d} h_j v\Big|$ over $v$  
with $\|v\|_{\tilde{X}^{\epsilon}_N (J_j)} = \| h_j \|_{\tilde{X}^{\epsilon}_N (J_j)^{\ast}}^{\frac{1 - \epsilon}{1 + \epsilon}}$. By duality between the spaces $\tilde{X}^{\epsilon}_N (J_j)$ and
$\tilde{X}^{\epsilon}_N (J_j)^*$, this supremum is exactly
$\| h_1\|_{\tilde{X}^{\epsilon}_N (J_1)^{\ast}}^{\frac{2}{1 +\epsilon}}$, and therefore 
\[
\| h_1 + h_2 \|_{\tilde{X}^{\epsilon}_N (I)^{\ast}}
\begin{ltae}
\geq \frac{\| h_1\|_{\tilde{X}^{\epsilon}_N (J_1)^{\ast}}^{\frac{2}{1 +\epsilon}} + \| h_2 \|_{\tilde{X}^{\epsilon}_N (J_2)^{\ast}}^{\frac{2}{1 +\epsilon}}}{\Big( \| h_1 \|_{\tilde{X}^{\epsilon}_N (J_1)^{\ast}}^{\frac{2}{1 +\epsilon}} + \| h_2 \|_{\tilde{X}^{\epsilon}_N (J_2)^{\ast}}^{\frac{2}{1 +\epsilon}} \Big)^{\frac{1 - \epsilon}{2}}}
\\
= \Big( \| h_1 \|_{\tilde{X}^{\epsilon}_N (J_1)^{\ast}}^{\frac{2}{1 +\epsilon}} + \| h_2 \|_{\tilde{X}^{\epsilon}_N (J_2)^{\ast}}^{\frac{2}{1 +\epsilon}} \Big)^{\frac{1 + \epsilon}{2}} \,,
\end{ltae}
\]
as desired. To prove
\eqref{eq:X-concat-sim},  it suffices to set
\[ \|v\|_{\tilde{X}^{\epsilon}_N (I)}^{\frac{2}{1 - \epsilon}} \eqd
\begin{ltae}
\| v \|_{L_t^{\infty} L_x^{\frac{2}{1 - \epsilon}} (I \times\R^d)}^{\frac{2}{1 - \epsilon}} + \| v \|_{L_t^{2 \frac{d + \sigma}{d}} L_x^{2 \frac{d + \sigma}{d}} (I \times \R^d)}^{\frac{2}{1 - \epsilon}}\\
+ \sum_{l = 1}^d \big(N^{- \frac{d - 1}{2}} \| v \|_{L_{e_l}^{\frac{2}{1 - \epsilon}, \infty} (I \times \R^d)}\big)^{\frac{2}{1 - \epsilon}}\\
+ \sum_{l = 1}^d \big(N^{\frac{\sigma - 1}{2}} \big\| \UP_{e_l}^{\Lap} v\big\|_{L_{e_l}^{\infty, \frac{2}{1 - \epsilon}} (I \times \R^d)}\big)^{\frac{2}{1 - \epsilon}}
\end{ltae}
\]
   and note that norms of $X^{\epsilon}_N (I)$ and $\tilde{X}^{\epsilon}_N (I)$ are equivalent with constants independent of $N$, possibly
   depending on $\epsilon$.
  To show \eqref{eq:X-concat-sim} let us first focus on directional norms. 
  Note that, it suffices to prove the following, more general estimate for $p, q \in [1, \infty]$ and $r \leq \min\{p, q\}$:
\[
\| v_1 + v_2 \|_{L_{e_l}^{p, q} (I \times
\R^d)}^{r}  \leq \| v_1 \|_{L_{e_l}^{p, q} (I \times \R^d)}^{r} + \| v_2 \|_{L_{e_l}^{p, q} (I \times \R^d)}^{r} \,.
\]

Indeed, if $p \leq q$, then $r \leq p$, and since $v_1$ and $v_2$ have disjoint supports and $(a + b)^s \leq a^s + b^s$ if $s \in [0, 1]$, we have
\[
\| v_1 + v_2 \|_{L_{e_l}^{p, q} (I \times \R^d)}^{r}
\begin{ltae}
= \Big(\int_{\R} \big(\| v_1 (x_1) + v_2 (x_1) \|^{q}_{L^{q}_{t, x'} ( I \times \R^d )}\big)^{\frac{p}{q}} \dd x_1\Big)^{\frac{r}{p}}
\\
= \Big(\int_{\R}  \big( \| v_1 (x_1) \|^q_{L^{q}_{t, x'}} + \| v_2 (x_1)\|^q_{L^{q}_{t, x'}} \big)^{\frac{p}{q}}  \dd x_1\Big)^{\frac{r}{p}}
\\
\leq \Big( \| v_1 \|_{L_{e_l}^{p, q} (I \times \R^d)}^{p} + \| v_2 \|_{L_{e_l}^{p, q} (I \times \R^d)}^{p} \Big)^{\frac{r}{p}}
\\
\leq  \| v_1 \|_{L_{e_{l}}^{p, q} (I \times \R^d)}^{r} + \| v_2 \|_{L_{e_l}^{p, q} (I \times \R^d)}^{r} \,.
\end{ltae}
\]
If $p \geq q$, then $r \leq q$, and  by the triangle inequality
\[
\| v_1 + v_2 \|_{L_{e_l}^{p, q} (I \times \R^d)}^{r} \begin{ltae}
= \Big(\int_{\R}  \big( \| v_1 (x_1) \|^q_{L^{q}_{t, x'}} + \| v_2 (x_1)\|^q_{L^{q}_{t, x'}} \big)^{\frac{p}{q}}  \dd x_1\Big)^{\frac{r}{p}} 
\\
=   \big\| \| v_1 (x_1) \|^q_{L^{q}_{t,x'}} + \| v_2 (x_1)\|^q_{L^{q}_{t, x'}} \big\|_{L^{\frac{p}{q}}_{x_1}}^{\frac{r}{q}}  
\\
\leq 
\Big(\big\| \| v_1 (x_1) \|^q_{L^{q}_{t,x'}} \big\|_{L^{\frac{p}{q}}_{x_1}} + \big\| v_2 (x_1)\|^q_{L^{q}_{t, x'}} \big\|_{L^{\frac{p}{q}}_{x_1}} \Big)^{\frac{r}{q}}  
\\
=  \big( \| v_1 \|_{L_{e_l}^{p, q} (I \times \R^d)}^{q} + \| v_2 \|_{L_{e_l}^{p, q} (I \times \R^d)}^{q} \big)^{\frac{r}{q}}
\\
\leq  \| v_1 \|_{L_{e_l}^{p, q} (I \times \R^d)}^{r} + \| v_2 \|_{L_{e_l}^{p, q} (I \times \R^d)}^{r} \,.
\end{ltae}
\] 
For isotropic norms we have for $r \leq p$
\[
\| v_1 + v_2 \|_{L_{t}^p L_x^{q} (I \times \R^d)}^{r} \begin{ltae}
\leq \Big\| \| v_1 \|_{L_x^q} + \|v_2\|_{L_x^q} \Big\|_{L_{t}^p(I)}^{r}
\\
= \Big(\int_{I}   \| v_1 (t) \|^p_{L^{q}_{x}} + \| v_2 (t)\|^p_{L^{q}_{x}}   \dd t\Big)^{\frac{r}{p}}
\\
\leq   \Big(\int_{I}   \| v_1 (t) \|^p_{L^{q}_{x}} \dd t\Big)^{\frac{r}{p}} + \Big( \int_{I}\| v_2 (t)\|^p_{L^{q}_{x}}   \dd t\Big)^{\frac{r}{p}}
\\
=  \| v_1 \|_{L_t^p L_{x}^{q} (I \times \R^d)}^{r} + \| v_2 \|_{L_t^p L_{x}^{q} (I \times \R^d)}^{r} \,.
\end{ltae}
\]

Finally, let us inductively construct the dyadic grid. One can start by setting
$I^0_0 = I$. Assume that $I^n_k = [t_k^n, t_{k + 1}^n)$ was already constructed and set  
$t^{n + 1}_{2 k} = t^n_k$. 
Since 
$t \mapsto \| h_1\|_{\tilde{X}^{\epsilon}_N ([t^n_k, t])^{\ast}}^{\frac{2}{1 + \epsilon}}$ is  continuous, 
there is $t^{n + 1}_{2 k + 1} \in [t^n_k, t^n_{k + 1}]$ so that
\[
\| h_1 \|_{\tilde{X}^{\epsilon}_N ([t^n_k, t^{n + 1}_{2 k +1}])^{\ast}}^{\frac{2}{1 + \epsilon}} = \frac{1}{2} \| h_1\|_{\tilde{X}^{\epsilon}_N ([t^n_k, t^n_{k + 1}])^{\ast}}^{\frac{2}{1+ \epsilon}} .
\]
By \eqref{eq:Xstar-concat-sim} it holds that
\[
\| h_1 \|_{\tilde{X}^{\epsilon}_N ([t^{n + 1}_{2 k + 1}, t^n_{k +1}])^{\ast}}^{\frac{2}{1 + \epsilon}} \leq \| h_1\|_{\tilde{X}^{\epsilon}_N ([t^n_k, t^n_{k + 1}])^{\ast}}^{\frac{2}{1+ \epsilon}}  - \| h_1 \|_{\tilde{X}^{\epsilon}_N ([t^n_k, t^{n +1}_{2 k + 1}])^{\ast}}^{\frac{2}{1 + \epsilon}} = \frac{1}{2} \| h_1\|_{\tilde{X}^{\epsilon}_N ([t^n_k, t^n_{k + 1}])^{\ast}}^{\frac{2}{1 + \epsilon}},
\]
and the iterations yield required required.
\end{proof*}

\begin{proof*}{Proof of \Cref{prop:main-linear-estimate}}
The solution $v$ is given by the Duhamel formula \eqref{eq:duhamel}. The
first claimed bound is the statement of
\Cref{lem:linear-evolution-christ-kiselev}. The latter bound follows by
adding up the estimates over $N \in 2^{\N}$ as per \eqref{eq:defn:X-norm}.
\end{proof*}

\section{Probabilistic Estimates}\label{sec:probabilistic}

The aim in this subsection is to establish an almost sure bound of the free
evolution with random initial data. The solution 
$e^{i t \Lap} f^{\omega}$ is measured in a Besov -type norm $Y^{S, \epsilon}
(I)$ with regularity index $S > 0$ (see 
\eqref{eq:defn:Y-norm} below), which as the norm $X^{\xs,
\epsilon} (I)$, contains components of Strichartz -type, directional
maximal components (see \Cref{lem:dir-maximal-unit-scale}), and directional
local smoothing (see \Cref{lem:dir-local-smoothing}) components. However,
because of the random nature of $f^{\omega}$ we prove a gain of derivatives
in the estimates compared to the  deterministic counterparts
In fact, we do not have any
loss of derivatives for the Strichartz -type norms with high spatial
integrability, and the directional maximal component shows a loss of only
$\frac{\sigma - 1}{2}$ derivatives compared to the $\frac{d - 1}{2}$
derivatives for deterministic evolution in the norm 
$X^{\xs,\epsilon} (I)$.

For $\epsilon \in [0, 1)$ we set
\begin{equation}\label{eq:defn:Y-norm}
\begin{ltae}
\|F\|_{Y^{S, \epsilon} (I)} \eqd \Big( \sum_{N \in 2^{\N}} N^{2 S} \|\LP_{N} F\|^2_{Y^{\epsilon}_N (I)} \Big)^{\frac{1}{2}},
\\
\|F\|_{Y^{\epsilon}_N (I)} \eqd \begin{ltae}
\| F \|_{L_t^{\infty} L_x^2 (I \times \R^d)} + \| F \|_{L_t^{2 \frac{d + \sigma}{\sigma}} L_x^{2 \frac{d + \sigma}{\sigma}} (I \times \R^d)}
\\
+ \sum_{l = 1}^d N^{- \frac{\sigma - 1}{2}} \| F \|_{L_{e_l}^{2, \frac{2}{\epsilon}} (I \times \R^d)}
\\
+ \sum_{l = 1}^d N^{\frac{\sigma - 1}{2}} \| \UP_{e_l}^{\Lap} F\|_{L_{e_l}^{\frac{2}{\epsilon}, 2} (I \times \R^d)} .
\end{ltae}
\end{ltae} 
\end{equation}
The main result of this section is the following.

\begin{proposition}\label{prop:Y-bounds}
Let $S > 0$ and $\epsilon \in (0, \frac{1}{2})$.
Given $f \in H^{S + | O (\epsilon) |} (\R^d)$ denote by $f^{\omega}$ the
randomization \eqref{eq:randomization} of $f$. There exist constants $C > 0$
and $c > 0$ such that for any $\lambda > 0$ it holds that
\[
\mathbb{P} \Big( \Big\{\omega \in \Omega : \big\|e^{- i t \Lap} f^{\omega} \big \|_{Y^{S, \epsilon} (I)} > \lambda\Big\} \Big) \leq C e^{- \frac{c \lambda^2}{\|f\|_{H_x^{S + | O (\epsilon) |} (\R^d)}^2}} . \]
In particular, almost surely we have
\[ 
\|e^{- i t \Lap} f^{\omega} \|_{Y^{S, \epsilon} (I)} < \infty\, . 
\]
\end{proposition}

The proof of the above statement is provided at the end of this section. We
begin with a version of \Cref{lem:dir-maximal} for unit
frequency scale operators $\QP_{n}$ defined in
\eqref{eq:proj-unit-scale}, where we already prove the gain of 
regularity. The proof is based on a maximal function estimate
for unit-scale frequency localized data. Note that the improvement 
in the local smoothing estimate \Cref{lem:dir-local-smoothing}
is not expected. 

\begin{lemma}\label{lem:dir-maximal-unit-scale}
Let $d > \sigma \geq 2$, let $\Lap$
satisfy \eqref{eq:symbol-bddness} and \eqref{eq:symbol-curvature}. Then for
all $l \in \{ 1, \ldots, d \}$ and all $n \in \Z^{d}$ we have
\begin{equation}\label{eq:dir-maximal-unit-scale}
\langle n\rangle^{- \frac{\sigma - 1}{2}} \Big\| e^{- i t \Lap} \QP_{n} f    \Big\|_{L_{e_l}^{2, \infty} (I \times \R^d)} \lesssim \|\QP_{n} f\|_{L_x^2 (\R^d)} .
\end{equation}
\end{lemma}

\begin{proof}
We mimic the proof of \Cref{lem:dir-maximal}, using a
$T T^{\ast}$ argument. Define
\[
\chi_n (\xi) \eqd \sum_{\substack{ k \in \Z^d\\ | k - n | \leq 10 d }} \psi (\xi - k)
\]
so that $\psi (\xi - n) \chi_n (\xi) = \psi (\xi - n)$. Analogously to
\eqref{eq:T-operator-max}, let
\[
T f (t, x) \eqd \int_{\R^d} e^{2 \pi i \xi x - i t \Lap (\xi)} \chi_n(\xi) \FT{f} (\xi) d \xi
\]
and as in \eqref{eq:T-operator-max}  we have
\[
\begin{ltae}
T T^{\ast} g (t, x) = \int_{\R \times \R^d} K_N (t - s, x - y) g (s, y) \dd s \dd y,
\\
K_n (t, x) \eqd \frac{1}{(2 \pi)^{2 d}} \int_{\R^d} e^{2 \pi i \xi\cdot x - i t \Lap (\xi)}
\chi_n^2 (\xi) \dd \xi,
\end{ltae}
\]
reducing our proof to showing that
\[
\| K_n \|_{L^{1, \infty}_{e_l} (\R \times \R^d)} \lesssim \langle n \rangle^{\sigma - 1}.
\]
Without loss of generality suppose $e_l = e_1$. Since $\chi_n$ is supported on a set of
measure of order 1 (instead of $N^d$ as was the case of
\Cref{lem:dir-maximal}) we have that
\[
| K_n (t, x) | \lesssim 1.
\]
The required bounds for $| n | \leq N_0$ for a large $N_0 \in 2^N$ are
obtained analogously to \Cref{lem:dir-maximal}. Let us concentrate on the case
$| n | > N_0$ for $N_0$ large enough so that \eqref{eq:symbol-bddness} and \eqref{eq:symbol-curvature} 
hold for $\xi \in \spt \chi_n$.
  
As in \eqref{eq:flfe} we obtain the desired result  when $| x_1 | \lesssim | n|^{\sigma - 1}|t|$:
\[
\Big\| \1_{|x_1 | \lesssim |n|^{\sigma - 1}|t|} K_n (t, x_1, x') \Big\|_{L^{1, \infty}_{e_l} (\R \times \R^d)} 
\begin{ltae}
\lesssim \Big\|\min \Big\{1, \frac{|n|^{(\sigma - 1)\frac{d}{\sigma}}}{|x_1|^{\frac{d}{\sigma}}}\Big\}\Big\|_{L^{1,\infty}_{e_l} (\R \times \R^d)}
\\
\lesssim\langle n \rangle^{\sigma - 1} .
\end{ltae}
\]
Next, we focus on the regime $| x_1 | \gtrsim | n |^{\sigma - 1}|t|$. For $\xi \in \spt(\chi_n)$,  
$| \xi | \approx |n|$, and by \eqref{eq:symbol-bddness} we have
$\big| \partial_{\xi_1} \Lap (\xi) \big| \lesssim | \xi |^{\sigma - 1} \approx | n  |^{\sigma - 1}$, and therefore  $\big| t \partial_{\xi_1}  \Lap (\xi) \big| \leq |t|| n |^{\sigma - 1} \lesssim | x_1 |$. 
Consequently, for the derivative of the phase we have
\[
|\partial_{\xi_1} \Big( 2 \pi (x_1 \xi_1 + x' \cdot \xi') - t \Lap (\xi)  \Big)| = |2 \pi x_1 - t \partial_{\xi_1}  \Lap (\xi)| \approx | x_1 | \,.
\]
Also, \eqref{eq:symbol-bddness} implies $| \partial_{\xi_1, \xi_1} \Lap (\xi) | \lesssim | \xi |^{\sigma - 2} \approx | n |^{\sigma - 2}$, and we deduce that
\[
\Big| \partial_{\xi_1} \Big( \frac{1}{2\pi x_1 - t \partial_{\xi_1} \Lap (\xi)} \Big) \Big|
=
\frac{\big| t \partial_{\xi_1, \xi_1}^2 \Lap (\xi) \big|}{\big| 2 \pi x_1 - t \partial_{\xi_1} \Lap (\xi) \big|^2} \lesssim \frac{|t| |n|^{\sigma - 2}}{|x_1 |^2}
\]
and similarly since $| \partial^3_{\xi_1, \xi_1, \xi_1} \Lap (\xi) | \lesssim | \xi |^{\sigma - 3} \approx | n |^{\sigma - 3}$ we deduce that
\[
\Big| \partial^2_{\xi_1, \xi_1} \Big( \frac{1}{2 \pi x_1 - t \partial_{\xi_1} \Lap (\xi)} \Big) \Big|
\lesssim
\frac{t | n |^{\sigma -  3}}{| x_1 |^2} + \frac{t^2 | n |^{2 \sigma - 4}}{| x_1 |^3} .
\]
Combining these estimates with $|x_1| \gtrsim |n|^{\sigma - 1}|t|$, $|n| \gg 1$, we obtain
\[
\begin{ltae}
\Big| \partial_{\xi_1} \Big( \frac{1}{2 \pi x_1 - t \partial_{\xi_1} \Lap (\xi)} \partial_{\xi_1} \big( \frac{\chi_N^2 (\xi)}{2 \pi x_1 - t  \partial_{\xi_1} \Lap (\xi)} \big) \Big) \Big|
\\ \qquad
\lesssim \frac{1}{| x_1 |^2} + \frac{|t| |n|^{\sigma - 2}}{|x_1  |^3} + \frac{t | n |^{\sigma - 3}}{| x_1 |^3} + \frac{t^2 | n |^{2 \sigma - 4}}{| x_1 |^4}
\\ \qquad
\lesssim \frac{1}{| x_1 |^2} \Big( 1 + \frac{|t||n|^{\sigma - 2}}{|x_1 |} + \frac{|t| | n |^{\sigma - 3}}{| x_1 |} + \frac{t^2 | n |^{2\sigma - 4}}{| x_1 |^2} \Big)
\lesssim \frac{1}{| x_1 |^2} \chi_N^2 (\xi) \,.
\end{ltae}
\]
And integration in $\xi$ gives
\[
\1_{| x_1 | \gtrsim | n |^{\sigma - 1}|t|} | K_n (t, x_1, x') | \lesssim |x_1 |^{- 2} .
\]
Combined with the fact that $| K_n (t, x_1, x') | < 1$ we obtain that
\[
\1_{| x_1 | \gtrsim | n |^{\sigma - 1}|t|} | K_n (t, x_1, x') | \lesssim\langle x_1 \rangle^{- 2}
\]
and thus
\[
\Big\| \1_{| x_1 | \gtrsim | n |^{\sigma - 1}|t|} K_n (t, x_1, x') \Big\|_{L^{1, \infty}_{e_1} (\R \times \R^d)} \lesssim 1.
\]
This completes the required assertion.
\end{proof}

The improved unit scale directional maximal bounds of
\Cref{lem:dir-maximal-unit-scale}, the directional local smoothing estimates
of \Cref{lem:dir-local-smoothing}, and the Strichartz estimates of
\Cref{lem:strichartz} together with the unit scale Bernstein inequality
\eqref{eq:unit-scale-bernstein}  provide us with almost sure boundedness of
$\big\| e^{- i t \Lap} f^{\omega} \big\|_{Y^{S, \epsilon} (I)}$ thanks
to the following large deviation estimates proved in {\cite[Lemma
3.1]{burqRandomDataCauchy2008a}}. For $\gamma > 0$,  we write $\|F\|_{L^{\gamma}_{\omega}}$ to
denote $(\mathbb{E}|F|^{\gamma})^{1 / \gamma}$. We remark that the following lemmas provide a bridge between the stochastic nature of the initial condition and deterministic estimates. 

\begin{lemma}
\label{lem:proba1}Let $\{g_n \}_{n = 1}^{\infty}$ be a sequence of real
valued, independent, zero mean, random variables associated with the
distributions $\{\mu_n \}_{n = 1}^{\infty}$ on a probability space $(\Omega,
\mf{A}, \mathbb{P})$. Assume that there exists $c > 0$ such that
\[ \Big| \int_{- \infty}^{\infty} e^{\gamma x} d \mu_n (x) \Big| \leq
e^{c \gamma^2}, \quad \text{for all } \gamma \in \R  \text{ and all } n \in \N . \]
Then, there exists $\alpha > 0$ such that for any $\lambda > 0$ and any
sequence $\{c_n \}_{n = 1}^{\infty} \in l^2 (\N ; \mathbb{C})$,
\[ \mathbb{P} \Big( \Big\{ \omega : \Big| \sum_{n = 1}^{\infty} c_n g_n
(\omega) \Big| > \lambda \Big\} \Big) \leq 2 e^{- \alpha
\frac{\lambda^2}{\sum_n |c_n |^2}} . \]
Hence, there exists $C > 0$ such that for $2 \leq \gamma < \infty$ and every
$\{c_n \}_{n = 1}^{\infty} \in l^2 (\N ; \mathbb{C})$
\[ \Big\| \sum_{n = 1}^{\infty} c_n g_n (\omega)
\Big\|_{L_{\omega}^{\gamma} (\Omega)} \leq C \sqrt{\gamma} \Big(
\sum_{n = 1}^{\infty} |c_n |^2 \Big)^{\frac{1}{2}} . \]
\end{lemma}

The proof of the next lemma is a slight modification of the proof in
{\cite[Lemma 4.5]{tzvetkovConstructionGibbsMeasure2009}}.

\begin{lemma}
\label{lem:proba2}Let $F$ be a real valued measurable function on a
probability space $(\Omega, \mf{A}, \mathbb{P})$. Suppose that there exists
$C_0 > 0$, $K > 0$, and $p_0 \geq 1$ such that for any $\gamma \geq
\gamma_0$ we have
\[ \|F\|_{L_{\omega}^{\gamma} (\Omega)} \leq \sqrt{\gamma} C_0 K. \]
Then, there exist $c > 0$ and $C_1 > 0$ depending on $C_0$ and $p_0$, but
independent of $K$, such that for every $\lambda > 0$,
\[ \mathbb{P} (\{\omega \in \Omega : |F (\omega) | > \lambda\}) \leq C_1
e^{- c \lambda^2 / K^2} . \]
In particular, we have
\[ \mathbb{P} (\{\omega \in \Omega : |F (\omega) | < \infty\}) = 1. \]
\end{lemma}

\begin{proof*}{Proof of \Cref{prop:Y-bounds}}
We estimate each of the Littlewood-Paley pieces $Y_N^{\epsilon} (\R)$ of
$Y^{S, \epsilon} (\R)$ separately. 
By Minkowski's inequality, for any function $F$ and any $\gamma \geq 2$ it
holds that
\[
\Big\| \| F \|_{Y^{S, \epsilon} (I)} \Big\|_{L_{\omega}^{\gamma}} \begin{ltae}
= \bigg\|  \Big( \sum_{N \in 2^{\N}} N^{2 S} \| \LP_{N} F  \|_{Y_N^{\epsilon} (I)}^2 \Big)^{\frac{1}{2}}\bigg\|_{L_{\omega}^{\gamma}} \\
\leq \bigg( \sum_{N \in 2^{\N}} N^{2 S} \| \| \LP_{N} F  \|_{Y_N^{\epsilon} (I)} \|_{L_{\omega}^{\gamma}}^2 \bigg)^{\frac{1}{2}} \,.
\end{ltae}
\]
In the following we substitute $F = e^{- i t \Lap} f^{\omega}$.
First, since $(\infty, 2)$ is an admissible pair, by Lemma \ref{lem:strichartz}, Lemma \ref{lem:proba1},
and Jensen inequality
\[
\Big\|\| \LP_{N} e^{- i t \Lap} f^{\omega} \|_{L^{\infty}_t L^2_x (I \times \R^d)}\Big\|_{L^\gamma_\omega} \leq \Big\|\| \LP_{N} f^{\omega} \|_{L^2_x (\R^d)}\Big\|_{L^\gamma_\omega}
\lesssim \| \LP_{N} f \|_{L^2_x (\R^d)} .
\]
Next, we bound the moments of the $L_t^{2 \frac{d + \sigma}{\sigma}} L_x^{2
\frac{d + \sigma}{\sigma}}$  norm.
For any $\gamma > \frac{2}{\epsilon} > 2 \frac{d + \sigma}{\sigma}$, by
Minkowski's inequality and \Cref{lem:proba1}
\[ \begin{ltae}
\Big\| \big\| \LP_{N} e^{- i t \Lap} f^{\omega}\big\|_{L_t^{2 \frac{d + \sigma}{\sigma}} L_x^{2 \frac{d + \sigma}{\sigma}} ( \R \times \R^d)} \Big\|_{L_{\omega}^{\gamma}}
\\
\begin{ltae}
\leq {\bigg\| \Big\| \sum_{n \in \Z^d} g_n (\omega) e^{- i t \Lap} \LP_{N} \QP_{n} f \Big\|_{L_{\omega}^{\gamma}} \bigg\| }_{L_t^{2 \frac{d+ \sigma}{\sigma}} L_x^{2 \frac{d + \sigma}{\sigma}} ( \R \times\R^d )}
\\
\lesssim \sqrt{\gamma} \bigg\| \Big( \sum_{n \in \Z^d} \big| e^{- i t \Lap} \LP_{N} \QP_{n} f \big|^2\Big)^{\frac{1}{2}} \bigg\|_{L_t^{2 \frac{d + \sigma}{\sigma}} L_x^{2 \frac{d + \sigma}{\sigma}} (\R \times \R^d )}
\\
\lesssim \sqrt{\gamma} \bigg( \sum_{n \in \Z^d} \big\| e^{- i t \Lap} \LP_{N} \QP_{n} f  \big\|^2_{L_t^{2 \frac{d + \sigma}{\sigma}} L_x^{2\frac{d + \sigma}{\sigma}} (\R \times \R^d)} \bigg)^{\frac{1}{2}} \,.
\end{ltae}
\end{ltae} \]
Note that $\Big( 2 \frac{d + \sigma}{\sigma}, \frac{2 d (d + \sigma)}{d (d
+ \sigma) - \sigma^2} \Big)$ is a $\Lap$-admissible pair of exponents and
\[
2 \frac{d + \sigma}{\sigma} > \frac{2 d (d + \sigma)}{d (d + \sigma) - \sigma^2},
\]
so we can use the unit scale Bernstein estimate
\eqref{eq:unit-scale-bernstein} coupled with the Strichartz estimate
\eqref{eq:strichartze1} to obtain
\[
\big\| e^{- i t \Lap} \LP_{N} \QP_{n} f \big\|^2_{L_t^{2 \frac{d +\sigma}{\sigma}} L_x^{2 \frac{d + \sigma}{\sigma}} (\R \times \R^d)}
\begin{ltae}
\lesssim \big\| e^{- i t \Lap} \LP_{N} \QP_{n} f \big\|^2_{L_t^{2 \frac{d+ \sigma}{\sigma}} L_x^{\frac{2 d (d + \sigma)}{d (d + \sigma) - \sigma^2}} (\R \times \R^d)}
\\
\lesssim \|\LP_{N} \QP_{n} f\|^2_{L_x^2 (\R^d)} .
\end{ltae}
\]

Next, we estimate the $L^{2, \frac{2}{\epsilon}}_{e_l} $ component of the
$Y_N^{\epsilon} (I)$ norm. As above, \Cref{lem:proba1} together with
Minkowski's inequality yields for any $\gamma > \frac{2}{\epsilon}$ that
\[ \begin{ltae}
\Big\| \sum_{l = 1}^d N^{- \frac{\sigma - 1}{2}} \big\| e^{- i t \Lap} \LP_{N} f^{\omega} \big\|_{L^{2, \frac{2}{\epsilon}}_{e_l} (I \times \R^d)} \Big\|_{L_{\omega}^{\gamma}}
\\
\qquad \begin{ltae}
\lesssim \sum_{l = 1}^d N^{- \frac{\sigma - 1}{2}} {\bigg\| \Big\| \sum_{n \in \Z^d} g_n (\omega) e^{- i t \Lap} \LP_{N} \QP_{n} f\Big\|_{L_{\omega}^{\gamma}} \bigg\| }_{L^{2, \frac{2}{\epsilon}}_{e_l} (I \times \R^d)}
\\
\lesssim \sqrt{\gamma} \sum_{l = 1}^d N^{- \frac{\sigma - 1}{2}}{\bigg\| \Big( \sum_{n \in \Z^d} \big| e^{- i t \Lap} \LP_{N} \QP_{n} f\big|^2 \Big)^{\frac{1}{2}} \bigg\| }_{L^{2, \frac{2}{\epsilon}}_{e_l} (I \times \R^d)}
\\
\lesssim \sqrt{\gamma} \sum_{l = 1}^d N^{- \frac{\sigma - 1}{2}} \Big( \sum_{n \in \Z^d} \big\|e^{- i t \Lap} \LP_{N} \QP_{n} f\big\|^2_{L^{2, \frac{2}{\epsilon}}_{e_l} (I \times \R^d)} \Big)^{\frac{1}{2}} .
\end{ltae}
\end{ltae} \]
From the Hölder inequality and Lemma \ref{lem:strichartz} follows
\[ \big\| e^{-it \Lap} \LP_{N} \QP_{n} f\big\| _{L^{2, 2}_{e_l} (I \times \R^d)}
\begin{ltae}
= \big\|e^{- i t \Lap} \LP_{N} \QP_{n} f\big\| _{L^2_t L^2_x (I \times \R^d)}
\\
\leq | T_0 |^{\frac{1}{2}} \|e^{- i t \Lap} \LP_{N} \QP_{n} f\|_{L^{\infty}_t L^2_x (I \times \R^d)}
\\
\lesssim | T_0 |^{\frac{1}{2}} \| \LP_{N} \QP_{n} f \|_{L^2( \R^d)}
\end{ltae} \]
and an interpolation with \Cref{lem:dir-maximal-unit-scale} gives us
\[ \langle n \rangle^{- \frac{\sigma - 1}{2}} \big\|e^{-it \Lap} \LP_{N} \QP_{n} f\big\|^2_{L^{2, \frac{2}{\epsilon}}_{e_l} (I \times \R^d)} \lesssim \langle n \rangle^{| O (\epsilon) |} \| \LP_{N} \QP_{n} f \|_{L^2 ( \R^d )} \]
and since $\LP_{N} \QP_{n}$ is nonzero only if $|n| \approx N$,
\[
\begin{ltae}
\Big\| \sum_{l = 1}^d N^{- \frac{\sigma - 1}{2}} \big\| e^{- i t \Lap}\LP_{N} f^{\omega} \big\|_{L^{2, \frac{2}{\epsilon}}_{e_l} (I \times \R^d)} \Big\|_{L_{\omega}^{\gamma}}
\\
\quad
\begin{ltae}
\lesssim \sqrt{\gamma} \sum_{l = 1}^d N^{| O (\epsilon) |} \Big( \sum_{n \in \Z^d} \|\LP_{N} \QP_{n} f\|^2_{L_x^2 (\R^d)}\Big)^{\frac{1}{2}}
\\
\lesssim \sqrt{\gamma} \sum_{l = 1}^d N^{| O (\epsilon) |} \|\LP_{N} f\|_{L_x^2 (\R^d)} .
\end{ltae}
\end{ltae}
\]

Similarly,
\[
\big\| \LP_{N} e^{- i t \Lap} f^{\omega} \big\|_{L_{e_l}^{2, 2} (I \times\R^d)}
\begin{ltae}
= \big\| \LP_{N} e^{- i t \Lap} f^{\omega} \big\|_{L_t^2 L^2_x (I\times \R^d)}
\\
\leq | T_0 |^{\frac{1}{2}} \big\| \LP_{N} e^{- i t \Lap} f^{\omega}\big\|_{L^{\infty}_t L^2_x (I \times \R^d)}
\\
\lesssim \| \LP_{N} f \|_{L^2_x (\R^d)}
\end{ltae} \]
and an interpolation with \eqref{eq:dir-local-smoothing} yields

\[
\big\| \LP_{N} e^{- i t \Lap} f^{\omega}\big\|_{L_{e_l}^{\frac{2}{\epsilon}, 2} (I \times \R^d)} \lesssim N^{-\frac{\sigma - 1}{2} + |O(\epsilon)|} \|\LP_{N} f \|_{L^2_x (\R^d)} .
\]

A combination of previous estimates and the definition of the Sobolev space
$H^{S + | O (\epsilon) |}$ give us
\[
\Big\|\big\| e^{-i t \Lap} f^{\omega} \big\|_{Y^S (\R)} \Big\|_{L^{\gamma}_{\omega}}
\lesssim \sqrt{\gamma} \|f\|_{H_x^{S + | O (\epsilon) |} (\R^d)},
\]
and we conclude by using \Cref{lem:proba2}.
\end{proof*}

\section{Nonlinear estimates}\label{sec:nonlinear}

In this section we estimate the expression
\[  |F + v|^2 (F + v)  \]
in the norm
$X^{\xs, \epsilon} (I)^{\ast}$, 
where $v$ is a fixed function with finite $X^{\xs} (I)$ norm and $F$ is controlled in the norm $Y^{S, \epsilon} (I)$.
In the next section, we set 
$F := e^{- i t \Lap}
f^{\omega}$.  

A Littlewood-Paley decomposition of both
$F$ and $v$ yields
\[ F = \sum_{N \in 2^{\N}} \LP_{N} F, \qquad v = \sum_{N \in 2^{\N}} \LP_{N} v, \]
so that
\[ |F + v|^2 (F + v) \begin{ltae}
     = \sum_{N_1, N_2, N_3 \in 2^{\N}} (\LP_{N_1} F +\LP_{N_1} v)
     \bar{(\LP_{N_2} F +\LP_{N_2} v)} (\LP_{N_3} F +\LP_{N_3} v)\\
     = \sum_{h_1, h_2, h_3 \in \{ F, v \}} \sum_{N_1, N_2, N_3 \in 2^{\N}}
    \LP_{N_1} h_1 \bar{ \LP_{N_2} h_2}\LP_{N_3} h_3 \\
     = \sum_{h_1, h_2, h_3 \in \{ F, v \}} \sum_{N_1, N_2, N_3, N_4 \in 2^{\N}}
    \LP_{N_4} (\LP_{N_1}h_1 \bar{ \LP_{N_2} h_2}\LP_{N_3} h_3)
     .
   \end{ltae} \]
For every assignment
$h_1, h_2, h_3 \in \{ F, v \}$ we estimate $ \LP_{N_4} (\LP_{N_1}h_1 \bar{ \LP_{N_2} h_2}\LP_{N_3} h_3)$ in 
$X^{\xs, \epsilon} (I)^{\ast}$. By duality is 
suffices to bound 
\[
\Big| \int_{I \times \R^d} (\LP_{N_1} h_1) \bar{(\LP_{N_2} h_2)}(\LP_{N_3} h_3) \bar{(\LP_{N_4} v_*)} \dd t \dd x   \Big|
\]
for any $v_*$ with $\|v_*\|_{X^{\epsilon}_{N_4} (I)} = 1$. To unify the notation, 
we define a number $\ms{A}_j$ and the norm $\| \cdot \|_{N_j, \epsilon} $ as 
\begin{equation}\label{eq:dfaj}
\begin{cases}
\ms{A}_j = S  \text{ and } \|\LP_{N_j} h_j\|_{N_j, \epsilon} = \|\LP_{N_j} F \|_{Y_{N_j}^{\epsilon} (I)}           & \text{ if } h_j = F             \\
\ms{A}_j = \xs  \text{ and } \|\LP_{N_j} h_j \|_{N, \epsilon} = \| \LP_{N_j} v \|_{X_{N_j}^{\epsilon} (I)}           & \text{ if } h_j = v             \\
\ms{A}_j = - \xs  \text{ and } \|\LP_{N_j} h_j \|_{N, \epsilon} = \|\LP_{N_j} v_{\ast} \|_{X_{N_j}^{\epsilon} (I)} & \text{ if }  h_j = v_{\ast} \,.
\end{cases}
\end{equation}
The main result of this section is the bound 
\begin{multline}
\Big| \int_{I \times \R^d} (\LP_{N_1} h_1) \bar{(\LP_{N_2} h_2)}(\LP_{N_3} h_3) \bar{(\LP_{N_4} h_4)} \dd t \dd x   \Big|
  \\
   \lesssim_{\tilde{\epsilon}} \max (N_1, N_2, N_3, N_4)^{-\tilde{\epsilon}} \prod_{j = 1}^4 N_j^{\ms{A}_j} \|\LP_{N_j} h_j \|_{N_j, \epsilon} .
   \label{eq:fixed-frequency-4linear}
\end{multline}
 for any small enough positive $\tilde{\epsilon} > 0$ and any
$h_1, h_2, h_3, h_4 \in \{ F, v, v_{\ast} \}$ with
$h_j = v_*$ for 
exactly $j \in \{1, \cdots, 4\}$.
Before proceeding, we state an application of  \eqref{eq:fixed-frequency-4linear}. 

\begin{corollary}\label{cor:main-nonlinear-estimate}
Let $h_1, h_2, h_3 \in \{ F, v \}$ and
suppose that \eqref{eq:fixed-frequency-4linear} holds for any small
$\tilde{\epsilon} > 0$ and any small $\epsilon > 0$, possibly
depending on $\tilde{\epsilon}$. Then,
\[
\| h_1 \bar{h}_2 h_3 \|_{X^{\xs + \tilde{\epsilon} / 8, \epsilon}(I)^{\ast}} \lesssim \prod_{j = 1}^3 \| h_j \|_{\ms{A}_j, \epsilon} ,
\]
where
\[
\| h_j \|_{\ms{A}_j, \epsilon} =
\begin{cases}
\| F \|_{Y^{S, \epsilon} (I)}   & \text{ if } h_j = F     \\
\| v \|_{X^{\xs, \epsilon} (I)} & \text{ if } h_j = v \,.
\end{cases}
\]
\end{corollary}

\begin{proof}
By the definition of $X^{s, \epsilon} (I)^{\ast}$ and  $X^{\epsilon}_N (I)^{\ast}$
\begin{align}
\| h_1 \bar{h}_2 h_3 \|^2_{X^{\xs + \tilde{\epsilon} / 8, \epsilon}
(I)^{\ast}}  = \sum_{N \in 2^{\N}} N^{2\xs + \frac{\tilde{\epsilon}}{4}} \|\LP_{N}(h_1 \bar{h}_2 h_3)\|_{X^\epsilon_N(I)^*}^2
\end{align}
and
\[
\|\LP_{N}(h_1 \bar{h}_2 h_3)\|_{X^\epsilon_N(I)^*}
\begin{ltae}
= \sup_{v_*} \bigg| \int_{I \times \R^d} \LP_{N}(h_1 \bar{h}_2 h_3) v_* \dd t \dd x \bigg|
\\
= \sup_{v} \bigg| \int_{I \times \R^d} h_1 \bar{h}_2 h_3 \LP_{N} v_* \dd t \dd x \bigg| \,,
\end{ltae}
\]
where the supremum is taken over $v_*$ with $\|v_*\|_{X^\epsilon_{N}(I)} \leq 1$.
Then, by   \eqref{eq:fixed-frequency-4linear} and 
$\max (N_1, N_2, N_3, N_4)^{-\tilde{\epsilon}} \leq (N_1N_2N_3N_4)^{-\tilde{\epsilon}/4}$ one has
\[
\|\LP_{N}(h_1 \bar{h}_2 h_3)\|_{X^\epsilon_N(I)^*}
\begin{ltae}
= \sup_{v_*} \bigg| \int_{I \times \R^d}  \sum_{N_1, N_2, N_3} (\LP_{N_1} h_1) \bar{(\LP_{N_2} h_2)}  (\LP_{N_3} h_3) \LP_{N} v_*  \dd t \dd x \bigg|
\\
\lesssim_{\tilde{\epsilon}} N^{-\xs - \frac{\tilde{\epsilon}}{4}} \sum_{N_1, N_2, N_3}  \prod_{j = 1}^3 N_j^{\ms{A}_j - \frac{\tilde{\epsilon}}{4}} \|\LP_{N_j} h_j  \|_{N_j, \epsilon} \,,
\end{ltae}
\]
and therefore by Cauchy-Schwarz inequality and by the summability of $\sum_{N \in 2^{\N}} N^{-\frac{\tilde{\epsilon}}{4}}$ we obtain
\[
\| h_1 \bar{h}_2 h_3 \|^2_{X^{\xs + \tilde{\epsilon} / 8, \epsilon} (I)^{\ast}}
\begin{ltae}
\lesssim_{\tilde{\epsilon}} \sum_{N \in 2^{\N}} N^{-\frac{\tilde{\epsilon}}{4}}  \prod_{j = 1}^3  \Big(\sum_{N_j} N_j^{\ms{A}_j -\frac{\tilde{\epsilon}}{4}} \|\LP_{N_j} h_j  \|_{N_j, \epsilon}\Big)^2
\\
\lesssim_{\tilde{\epsilon}}\prod_{j = 1}^3  \sum_{N_j} N_j^{2\ms{A}_j} \|\LP_{N_j} h_j \|_{N_j, \epsilon}^2 \lesssim \prod_{j = 1}^3 \| h_j \|_{\ms{A}_j, \epsilon}^2\,,
\end{ltae}
\]
as desired.
\end{proof}

\begin{proposition}\label{prop:fixed-frequency-4linear}
If
\[ \frac{d - \sigma}{2} \geq S > S_{\min} \eqd \frac{(d - \sigma)}{2}
\begin{cases}
\frac{1}{3} & \sigma \geq \frac{d + 2}{3}
\\
\frac{d + 1 - 2 \sigma}{(d - 1)} &  \sigma \leq \frac{d + 2}{3}\,,
\end{cases}
\]
then there exists $\xs > \frac{d - \sigma}{2}$ and $\epsilon > 0$ sufficiently small
such that bound \eqref{eq:fixed-frequency-4linear} holds.
\end{proposition}

The proof of \Cref{prop:fixed-frequency-4linear} relies on
the following bilinear estimates.

\begin{lemma}\label{lem:bilinear-2-2-bound}
Let $N_+, N_- \in 2^{\N}$ with $N_+ \gtrsim N_-$. The following bilinear estimates hold for any two functions $h_+, h_- \st I \times \R^d \rightarrow \C$:
\begin{align}
\label{eq:XX-YY-XY-1}
 & \begin{ltae}
\|\LP_{N_+} h_+\LP_{N_-} h_- \|_{L^2_t L^2_x ( I \times \R^d )}
\\ \qquad\leq N_+^{| O (\epsilon) |} N_+^{- \frac{\sigma - 1}{2}} N_-^{\frac{d- 1}{2}} \|\LP_{N_+} h_+ \|_{X^{\epsilon}_{N_+} (I)} \|\LP_{N_-} h_-\|_{X^{\epsilon}_{N_-} (I)},
\end{ltae}
\\
\label{eq:XX-YY-XY-2}
 & \begin{ltae}
\|\LP_{N_+} h_+\LP_{N_-} h_- \|_{L^2_t L^2_x ( I \times \R^d )}
\\ \qquad \leq N_+^{| O (\epsilon) |}  N_+^{- \frac{\sigma - 1}{2}} N_-^{\frac{\sigma - 1}{2}} \|\LP_{N_+} h_+ \|_{Y_{N_+}^{\epsilon} (I)} \|\LP_{N_-} h_- \|_{Y_{N_-}^{\epsilon} (I)},
\end{ltae}
\\
\label{eq:XX-YY-XY-3}
 & \begin{ltae}
\|\LP_{N_+} h_+\LP_{N_-} h_- \|_{L^2_t L^2_x ( I \times \R^d )}
\\ \qquad\leq N_+^{| O (\epsilon) |}  N_+^{- \frac{\sigma - 1}{2}} N_-^{\frac{\sigma - 1}{2}} \|\LP_{N_+} h_+ \|_{X_{N_+}^{\epsilon} (I)} \|\LP_{N_-} h_- \|_{Y_{N_-}^{\epsilon} (I)},
\end{ltae}
\end{align}
and
\begin{equation}\label{eq:YX}
\begin{ltae}
\|\LP_{N_+} h_+\LP_{N_-} h_- \|_{L^2_t L^2_x ( I \times \R^d )}
\\
\qquad \leq N_+^{| O (\epsilon) |}  N_+^{- \frac{\sigma - 1}{2} \theta} N_-^{\frac{d - 1}{2} \theta} \|\LP_{N_+} h_+ \|_{Y_{N_+}^{\epsilon}(I)} \|\LP_{N_-} h_- \|_{X_{N_-}^{\epsilon} (I)}
\end{ltae}
\end{equation}
for any $\theta \in [0, 1]$.
\end{lemma}

\begin{proof}
The proof of all the  bounds follows from the Hölder
inequality and the definition of the norms $X^{\epsilon}_{N_{\pm}}$
and $Y^{\epsilon}_{N_{\pm}}$ and Lemma \ref{lem:bernstein}. To prove \eqref{eq:XX-YY-XY-1}, notice that
\[ 
\begin{ltae}
\|\LP_{N_+} h_+\LP_{N_-} h_- \|_{L^{\frac{2}{1 - \epsilon}}_t  L^{\frac{2}{1 - \epsilon}}_x ( I \times \R^d )}
\\
\qquad \begin{ltae}
\leq \sum_{l = 1}^d \Big\|\LP_{N_+} U^{\Lap}_{e_l} h_+\LP_{N_-} h_- \Big\|_{L^{\frac{2}{1 - \epsilon}}_t L^{\frac{2}{1 - \epsilon}}_x ( I \times \R^d)}
\\
= \sum_{l = 1}^d \Big\|\LP_{N_+} U^{\Lap}_{e_l} h_+\LP_{N_-} h_- \Big\|_{L_{e_l}^{\frac{2}{1 - \epsilon}, \frac{2}{1 - \epsilon}} ( I \times \R^d )}
\\
\leq \sum_{l = 1}^d \Big\|\LP_{N_+} \UP_{e_l}^{\Lap} h_+  \Big\|_{L_{e_l}^{\infty, \frac{2}{1 - \epsilon}} (I \times\R^d)} \|\LP_{N_-} h_- \|_{L_{e_l}^{\frac{2}{1 - \epsilon}, \infty} (I \times \R^d)}
\\
\leq N_+^{- \frac{\sigma - 1}{2}} N_-^{\frac{d - 1}{2}} \|\LP_{N_+} h_+ \|_{X^{\epsilon}_{N_+} (I)} \|\LP_{N_-} h_- \|_{X^{\epsilon}_{N_-} (I)}
\end{ltae}
\end{ltae}
\]
and \eqref{eq:XX-YY-XY-1} follows after an interpolation with
\[
\begin{ltae}
\|\LP_{N_+} h_+\LP_{N_-} h_- \|_{L^\frac{3}{2} L^\frac{3}{2}_x ( I \times \R^d)}
\\
\qquad\begin{ltae}
\leq \|\LP_{N_+} h_+ \|_{L^3_t L^3_x ( I \times \R^d )} \|\LP_{N_-} h_- \|_{L^3_t L^3_x ( I \times \R^d )}
\\
\lesssim \|\LP_{N_+} h_+ \|_{L^\infty_t L^3_x ( I \times \R^d )} \|\LP_{N_-} h_- \|_{L^\infty_t L^3_x ( I \times \R^d )}
\\
\lesssim N_+^{\frac{1-3\epsilon}{6}} N_-^{\frac{1-3\epsilon}{6}} \|\LP_{N_+} h_+ \|_{L^\infty_t L^{\frac{2}{1-\epsilon}}_x ( I \times \R^d )} \|\LP_{N_-} h_- \|_{L^\infty_t L^{\frac{2}{1-\epsilon}}_x ( I \times \R^d )}
\end{ltae}
\end{ltae}
\]
where we used Bernstein inequality, Lemma \ref{lem:bernstein}.
To prove \eqref{eq:XX-YY-XY-2}, we use 
\[
\begin{ltae}
\|\LP_{N_+} h_+\LP_{N_-} h_- \|_{L^{\frac{2}{1 +  \epsilon}}_t  L^{\frac{2}{1 +  \epsilon}}_x ( I \times \R^d )}
\\ \qquad \begin{ltae}
\leq \sum_{l = 1}^d \Big\|\LP_{N_+} U^{\Lap}_{e_l} h_+\LP_{N_-} h_- \Big\|_{L^{\frac{2}{1 +  \epsilon}}_t L^{\frac{2}{1 +  \epsilon}}_x ( I \times \R^d )}
\\
= \sum_{l = 1}^d \Big\|\LP_{N_+} U^{\Lap}_{e_l} h_+\LP_{N_-} h_- \Big\|_{L_{e_l}^{\frac{2}{1 +  \epsilon}, \frac{2}{1 +  \epsilon}} ( I \times \R^d )}
\\
\leq \sum_{l = 1}^d \Big\|\LP_{N_+} \UP_{e_l}^{\Lap} h_+  \Big\|_{L_{e_l}^{\frac{2}{\epsilon}, 2} (I \times \R^d)} \| \LP_{N_-} h_- \|_{L_{e_l}^{2, \frac{2}{\epsilon}} (I \times \R^d)}\\
\leq N_+^{- \frac{\sigma - 1}{2}} N_-^{\frac{\sigma - 1}{2}} \|\LP_{N_+}  h_+ \|_{Y^{\epsilon}_{N_+} (I)} \|\LP_{N_-} h_-  \|_{Y^{\epsilon}_{N_-} (I)}\,,
\end{ltae}
\end{ltae}
\]
 interpolated with 
\[
\|\LP_{N_+} h_+\LP_{N_-} h_- \|_{L^{\infty}_t L^{\infty}_x ( I \times \R^d )}
\begin{ltae}
\leq \|\LP_{N_+} h_+  \|_{L^{\infty}_t L^{\infty}_x ( I \times \R^d)} \|\LP_{N_-} h_- \|_{L^{\infty}_t L^{\infty}_x ( I \times \R^d )}
\\
\leq N_+^{\frac{d}{2}} N_-^{\frac{d}{2}} \|\LP_{N_+} h_+ \|_{L^{\infty}_t L^2_x ( I \times \R^d )} \|\LP_{N_-} h_- \|_{L^{\infty}_t L^2_x ( I \times \R^d )}
\\
\leq N_+^{\frac{d}{2}} N_-^{\frac{d}{2}} \|\LP_{N_+} h_+\|_{Y^{\epsilon}_{N_+} (I)} \|\LP_{N_-} h_- \|_{Y^{\epsilon}_{N_-}(I)} \,,
\end{ltae}
\]
where we again used Lemma \ref{lem:bernstein}.
Next,  \eqref{eq:XX-YY-XY-3} follows from 
\[
\|\LP_{N_+} h_+\LP_{N_-} h_- \|_{L^2_t L^2_x ( I \times \R^d )}
\begin{ltae}
\leq \sum_{l = 1}^d \Big\|\LP_{N_+} \UP^{\Lap}_{e_l} h_+\LP_{N_-} h_-\Big\|_{L_{e_l}^{2, 2} ( I \times \R^d)}
\\
\leq \sum_{l = 1}^d \Big\|\LP_{N_+} \UP^{\Lap}_{e_l} h_+ \Big\|_{L_{e_l}^{\infty, \frac{2}{1-\epsilon}} (I \times \R^d)} \| \LP_{N_-} h_- \|_{L_{e_l}^{2, \frac{2}{\epsilon}} (I \times \R^d)}
\\
\leq \sum_{l = 1}^d N_+^{- \frac{\sigma - 1}{2}} N_-^{\frac{\sigma -1}{2}} \|\LP_{N_+} h_+ \|_{X^{\epsilon}_{N_+} (I)} \|\LP_{N_-} h_- \|_{Y^{\epsilon}_{N_-} (I)} .
\end{ltae} \]
Finally, \eqref{eq:YX} is a consequence of an interpolation between 
\begin{align*}
\|\LP_{N_+} h_+\LP_{N_-} h_- \|_{L^2_t L^2_x ( I \times \R^d )}
 & \leq \|\LP_{N_+} h_+\|_{L_{e_l}^{\frac{2}{\epsilon}, 2} (I \times\R^d)} \|\LP_{N_-} h_- \|_{L_{e_l}^{\frac{2}{1-\epsilon}, \infty} (I \times \R^d)}
\\
 & \leq  N_+^{- \frac{\sigma - 1}{2}} N_-^{\frac{d - 1}{2}} \|\LP_{N_+} h_+ \|_{Y_{N_+}^{\epsilon} (I)} \|\LP_{N_-} h_- \|_{X_{N_-}^{\epsilon} (I)}
\end{align*}
and
\[
\begin{ltae}
\| \LP_{N_+} h_+\LP_{N_-} h_- \|_{L^2_t L^2_x ( I \times \R^d )}
\\ \qquad
\begin{ltae}
\leq \|\LP_{N_+} h_+  \|_{L_t^{2 \frac{d + \sigma}{\sigma}} L_x^{2 \frac{d + \sigma}{\sigma}} (I \times \R^d)} \|\LP_{N_-} h_- \|_{L_t^{2 \frac{d +\sigma}{d}} L_x^{2 \frac{d + \sigma}{d}} (I \times \R^d)}
\\
\leq   \|\LP_{N_+} h_+\|_{Y_{N_+}^{\epsilon} (I)} \|\LP_{N_-} h_- \|_{X_{N_-}^{\epsilon}(I)} \,.
\end{ltae}
\end{ltae}
\]
\end{proof}

\begin{proof*}{Proof of \Cref{prop:fixed-frequency-4linear}}
For ease of notation we drop the complex conjugation, as it does not
influence our estimates.  Furthermore, by symmetry, we assume that
\[
N_1 \geq N_2 \geq N_3 \geq N_4.
\]
Note that unless $N_1 \lesssim N_1$ one has $N_2 + N_3 + N_4 < 2^{- 10} N_1$ and thus
\[
\bar{(\LP_{N_2} h_2)} (\LP_{N_3} h_3) \bar{(\LP_{N_4} h_4)}
\]
is supported on $| \xi | < 2^{- 2} N_1$ in Fourier space. This implies
\[
\Big| \int_{I \times \R^d} (\LP_{N_1} h_1) \bar{(\LP_{N_2} h_2)} (\LP_{N_3}  h_3) \bar{(\LP_{N_4} h_4)} \dd t \dd x \Big| = 0 \,,
\]
and therefore we restrict to proving bound
\eqref{eq:fixed-frequency-4linear} only when
\[
N_1 \approx N_2 \geq N_3 \geq N_4 .
\]
We consider all the cases
$(h_1, h_2, h_3, h_4) \in \{ F, v, v_{\ast} \}^4$ with $v_{\ast}$ appearing
exactly once. Henceforth, we assume that
$\|\LP_{N_j} h_j \|_{N_j, \epsilon} = N_j^{\ms{- A}_j}$,
$j \in \{ 1, 2, 3, 4 \}$, set
\[
\xs = \frac{d - \sigma}{2} + \tilde{\epsilon} + | O(\epsilon) | \,,
\]
and prove the bound
\[
\Lambda (h_1, h_2, h_3, h_4) := \Big| \int_{I \times \R^d} (\LP_{N_1} h_1) (\LP_{N_2}
h_2) (\LP_{N_3} h_3) (\LP_{N_4} h_4) \dd t \dd x \Big| \lesssim N_1^{- \tilde{\epsilon}}\,.
\]
In general, by Cauchy-Schwarz inequality
\[ \Lambda (h_1, h_2, h_3, h_4) \leq
\big \|\LP_{N_1} h_1\LP_{N_4} h_4\big\|_{L^2_t  L^2_x ( I \times \R^d )}
\big\| \LP_{N_2} h_2\LP_{N_3} h_3 \big\|_{L^2_t L^2_x ( I \times \R^d )}
\]
and for each term on the right-hand side we use $N_1 \approx N_2$ and \Cref{lem:bilinear-2-2-bound} to obtain
\begin{equation}\label{eq:fbfce}
\Lambda (h_1, h_2, h_3, h_4) \leq  \prod_{j = 1}^4 N_j^{\ms{B}_j - \ms{ A}_j} \lesssim N_1^{|O(\epsilon)|} N_1^{\ms{B}_1 + \ms{B}_2  - \ms{ A}_1 - \ms{ A}_2}
N_3^{\ms{B}_3  - \ms{ A}_3}   N_4^{\ms{B}_4  - \ms{ A}_4} \,, 
\end{equation}   
where $\ms{ A}_j$ was defined in \eqref{eq:dfaj} and $\|\LP_{N_j} h_j \|_{N_j, \epsilon} = N_j^{\ms{- A}_j}$, and $\ms{B}_j$ is an appropriate exponent originating in estimates in 
\Cref{lem:bilinear-2-2-bound}. Let us provide details for various cases. 

\noindent
\textbf{Case 1.} Assume there exists $v$ with higher frequency than $v_*$.
More precisely, assume there is $j < k$ such that $h_j = v$ and $h_k = v_*$. Of course by the symmetry in \eqref{eq:fbfce} and $N_1 \approx N_2$
we also allow $j = 2$ and $k = 1$. 

After an exchange of $h_1$ for $h_2$ if necessary, we can assume $(j, k) \not \in \{(1, 4), (2, 3)\}$. 
 Then, $\ms{ A}_j = \xs$ and $\ms{ A}_k = -\xs$, which implies $\ms{ A}_j + \ms{ A}_k = 0$.
Observe that if $k \in \{1, 2\}$, then $j \in \{1, 2\}$ and by $N_1 \approx N_2$ and 
$\ms{ A}_j + \ms{ A}_k = 0$, one has 
\begin{equation}\label{eq:sscn2}
N_j^{\ms{B}_j-\ms{ A}_j} N_k^{\ms{B}_k-\ms{ A}_k} \lesssim N_1^{\ms{B}_j + \ms{B}_k} \,.
\end{equation}
If $k \in \{3, 4\}$, then $\ms{B}_k >0$, and therefore
$\ms{B}_k - \ms{A}_k > 0$. Since $N_j \geq N_k$ and
$\ms{ A}_j + \ms{ A}_k = 0$, one has
\begin{equation}\label{eq:sscn3}
N_j^{\ms{B}_j-\ms{ A}_j} N_k^{\ms{B}_k-\ms{ A}_k} \lesssim N_j^{\ms{B}_j + \ms{B}_k}  \lesssim N_1^{\ms{B}_j + \ms{B}_k}\,,
\end{equation}
where in the last inequality we used that either $j \in \{1, 2\}$ or $j > 2$ which implies 
$\ms{B}_j, \ms{B}_k \geq 0$. 

If $(h_1, h_4) \not \in \{(v, v), (v_*, v), (v, v_*) \}$, then the exponents of $N_+$ and $N_-$
in Lemma \ref{lem:bilinear-2-2-bound} add up to zero (we take $\theta=0$ in \eqref{eq:YX}), which in our notation means 
\begin{equation}\label{eq:coeob}
\ms{B}_1 + \ms{B}_4 = 0 \,.
\end{equation}
Similarly, \eqref{eq:coeob} holds for indices $(1, 4)$ replaced by
$(2, 3)$. For any $l$ with $l \neq j, k$, we have
$-\ms{A}_l \leq 0 < \tilde{\epsilon} + |O(\epsilon)|$.  If
$l \geq 3$, $\ms{B}_l > 0$, and we choose small
$\epsilon, \tilde{\epsilon} > 0$ such that
$\ms{B}_l - \tilde{\epsilon} - |O(\epsilon)| > 0$. If
$l \leq 2$, we choose any small $\epsilon, \tilde{\epsilon} > 0$ so that
$\ms{A}_l \geq \tilde{\epsilon} + |O(\epsilon)|$.  Then,
$N_l^{\ms{B}_l - \ms{A}_l} \leq N_l^{\ms{B}_l - \tilde{\epsilon} - |O(\epsilon)|} \lesssim N_1^{\ms{B}_l -
  \tilde{\epsilon} - |O(\epsilon)|}$, and after using $\ms{B}_1 + \cdots + \ms{B}_4 =0$ we obtain
\begin{align*}
\Lambda (h_1, h_2, h_3, h_4) &\leq  N_1^{|O(\epsilon)|} N_1^{ \sum_j \ms{B}_j - \tilde{\epsilon} - |O(\epsilon)|} 
 \lesssim N_1^{-\tilde{\epsilon}}
\end{align*}
as desired. 

If $(h_1, h_4) = (v, v)$, then we set $j = 1$ (higher frequency than $v_*$) and by Lemma \ref{lem:bilinear-2-2-bound},
$\ms{B}_1 = -\frac{\sigma - 1}{2}$ and $\ms{B}_4 - \ms{ A}_4 = \frac{d-1}{2} - \xs \leq \frac{\sigma - 1}{2} - \tilde{\epsilon} - | O(\epsilon) |$. Thus, we transformed the 
problem to $(h_1, h_4) = (v, F)$ from above (after $A_l$ was replaced by $\tilde{\epsilon} + |O(\epsilon)|$). 

Assume $(h_1, h_4) = (v, v_*)$. If 
$h_2 = F$, we exchange $h_1$ with $h_2$ (recall $N_1 \approx N_2$), to obtain 
$(h_1, h_4) \neq (v, v_*)$. Thus, we can assume $h_2 = v$. 
Then, by \eqref{eq:sscn2} and \eqref{eq:sscn3} with $(j, k) = (1, 4)$ and $N_3 \leq N_1$
 we have 
\begin{equation}
N_1^{\ms{B}_1 - \ms{ A}_1 + \ms{B}_2 - \ms{ A}_2}
N_3^{\ms{B}_3 - \ms{ A}_3} N_4^{\ms{B}_4 - \ms{ A}_4} \lesssim 
N_1^{\ms{B}_1 + \ms{B}_2 + \ms{B}_4  - \xs + \max\{\ms{B}_3 - \ms{ A}_3, 0\}} \,.  
\end{equation} 
In addition, 
\begin{equation}
   \ms{B}_1 + \ms{B}_4  - \xs < -\frac{\sigma - 1}{2} + \frac{d - 1}{2} - \frac{d - s}{2} = 0 
\end{equation}
and if $h_3 = F$, then $\ms{B}_3 - \ms{ A}_3 = \frac{\sigma - 1}{2} - S \leq \frac{\sigma - 1}{2}$ and if $h_3 = v$, then $\ms{B}_3 - \ms{ A}_3 = \frac{d - 1}{2} - \xs <  \frac{\sigma - 1}{2}$. Since $\ms{B}_2 = -  \frac{\sigma - 1}{2}$, we obtain 
\begin{equation}
N_1^{\ms{B}_1 - \ms{ A}_1 + \ms{B}_2 - \ms{ A}_2}
N_3^{\ms{B}_3 - \ms{ A}_3} N_4^{\ms{B}_4 - \ms{ A}_4} \lesssim 
N_1^{- \tilde{\epsilon} - |O(\epsilon)|} \,,
\end{equation} 
for any small $\epsilon, \tilde{\epsilon} > 0$ such that $\tilde{\epsilon} + |O(\epsilon)| < \xs - \frac{d - s}{2}$. 

Finally, if $(h_1, h_4) = (v_*, v)$, then there is $v$ with higher frequency than $v_*$, implying $h_2 = v$. 
As above, for $l \in \{3, 4\}$, we have $\ms{B}_l - \ms{ A}_l < \frac{\sigma - 1}{2}$, and
therefore
\begin{equation}
N_1^{\ms{B}_1 - \ms{ A}_1 + \ms{B}_2 - \ms{ A}_2}
N_3^{\ms{B}_3 - \ms{ A}_3} N_4^{\ms{B}_4 - \ms{ A}_4} \lesssim 
N_1^{\ms{B}_1 + \ms{B}_2 + \sigma - 1} 
\end{equation} 
and the assertion follows as above, because $\ms{B}_1 = \ms{B}_2 = - \frac{\sigma - 1}{2}$.

In the rest of the proof we only consider the cases, where $v_*$ has higher frequency than any $v$. Hence, there can be at most two functions $v$, since $N_1 \approx N_2$. 

\textbf{Case 2.} Assume that among $(h_j)_{j = 1}^4$ are two functions $v$. Specifically, 
assume $(h_1, h_2, h_3, h_4) = (v_{\ast}, F, v, v)$ or $(F, v_{\ast}, v, v)$. Since, we can interchange $N_1$ and $N_2$, 
we can just treat  $(h_1, h_2, h_3, h_4) = (v_{\ast}, F, v, v)$. By \eqref{eq:fbfce}, Lemma \ref{lem:bilinear-2-2-bound}, and $\xs \leq \frac{d - 1}{2}$ with $N_3 \geq N_4$ one has 
for any $\theta \in [0, 1]$
  \[ | \Lambda (v_{\ast}, F, v, v) | \begin{ltae}
       \lesssim N_1^{| O (\epsilon) | + \xs - \frac{\sigma - 1}{2}} N_2^{-
       S - \frac{\sigma - 1}{2} \theta} N_3^{\mf{- s} + \frac{d - 1}{2}
       \theta} N_4^{- \xs + \frac{d - 1}{2}}\\
       \lesssim N_1^{| O (\epsilon) | + \xs - \frac{\sigma - 1}{2} - S -
       \frac{\sigma - 1}{2} \theta} N_3^{- 2 \xs + \frac{d - 1}{2} + \frac{d -
       1}{2} \theta}\\
       \lesssim N_1^{| O (\epsilon) | + \tilde{\epsilon} + \frac{d + 1 -
       2 \sigma}{2} - S - \frac{\sigma - 1}{2} \theta} N_3^{- \frac{d + 1 - 2
       \sigma}{2} + \frac{d - 1}{2} \theta}
     \end{ltae} \]
  
  If $d + 1 - 2\sigma \leq 0$, then set $\theta = 0$ and drop $N_3$ term, since it has a negative power to obtain
  \[ | \Lambda (v_{\ast}, F, v, v) | \lesssim N_1^{| O (\epsilon) | +
     \tilde{\epsilon} - S}
      \lesssim N_1^{- \tilde{\epsilon}} \]
  as long as \ $S > | O (\epsilon) | + 2 \tilde{\epsilon}$. If $\frac{d
  + 1}{2} > \sigma$, then set $\theta = \frac{d + 1 - 2 \sigma}{d - 1}$ (note
  that $\theta \in (0, 1)$) to obtain
  \[ | \Lambda (v_{\ast}, F, v, v) | \lesssim N_1^{| O (\epsilon) | +
     \tilde{\epsilon} + \frac{d + 1 - 2 \sigma}{2} - S - \frac{\sigma -
     1}{2} \frac{d + 1 - 2 \sigma}{d - 1}} \lesssim N_1^{-
     \tilde{\epsilon}} \]
  as long as 
  \begin{equation}
  S> | O (\epsilon) | + 2 \tilde{\epsilon} + \Big( \frac{d + 1}{2} - \sigma \Big) \Big( \frac{d - \sigma}{d - 1} \Big) \,.
  \end{equation} 
  
 \textbf{Case 3.} Assume there is exactly one $j$ with $h_j = v$. 
 We distinguish two cases.
 
  \textbf{Case 3a.} Suppose $(h_1, h_2, h_3, h_4) \in \{( v_*, F, F, v), (F, v_*, F,  v), (F, F, v_*, v)\}$, that is, $h_4 = v$. 
  Since, we can interchange $N_1$ and $N_2$, we  only treat option $( v_*, F, F, v)$
  and $(F, F, v_*, v)$.
 To treat $( v_*, F, F, v)$, we apply \eqref{eq:YX} with parameter $\theta \in [0, 1]$ to be determined below and use $N_1 \approx N_2$
to have
  \[ | \Lambda (v_{\ast}, F,  F, v) | \begin{ltae}
       \lesssim N_1^{| O (\epsilon) |  - S -
       \theta \frac{\sigma - 1}{2} + \xs - \frac{\sigma - 1}{2}}
       N_3^{- S + \frac{\sigma - 1}{2}} N_4^{- \xs + \theta
       \frac{d - 1}{2}}
     \end{ltae} \]
Set $\theta = \frac{d - \sigma}{d - 1}$ and since $N_3, N_4 \leq N_1$, and $\xs > \frac{d - \sigma}{2}$, 
  \[ | \Lambda (v_{\ast}, F, F, v) | \lesssim N_1^{| O (\epsilon) |
     + \tilde{\epsilon} -  S - \frac{d-\sigma}{d-1}\frac{\sigma - 1}{2} + \frac{d - \sigma}{2} - \frac{\sigma - 1}{2} + \max\{- S + \frac{\sigma - 1}{2}, 0\}}
     \lesssim N_1^{- \tilde{\epsilon}} 
\]
as long as 
   \[ S >    \frac{d - \sigma}{2}  -\frac{d-\sigma}{d-1}\frac{\sigma - 1}{2}  - \frac{\sigma - 1}{2}
  + 2 \tilde{\epsilon} + | O
     (\epsilon) |  \]
and
\[
S >  \frac{1}{2} \Big(  \frac{d - \sigma}{2}  - \frac{(d-\sigma)}{(d-1)}\frac{\sigma - 1}{2} \Big)  + 2 \tilde{\epsilon} + | O(\epsilon) | =  \frac{(d-\sigma)^2}{4(d - 1)} + 2 \tilde{\epsilon} + | O(\epsilon) | \,.
\]
If $(h_1, h_2, h_3, h_4) =  (F, F, v_*, v)$, then similarly
\[ | \Lambda (F,  F, v_{\ast},  v) | \begin{ltae}
\lesssim N_1^{| O (\epsilon) |  - 2S - \theta_1 \frac{\sigma - 1}{2} - \theta_2 \frac{\sigma - 1}{2}} N_3^{ \xs + \theta_1 \frac{d - 1}{2}} N_4^{- \xs + \theta_2  \frac{d - 1}{2}}
\end{ltae} \]
and by choosing $\theta_1 = 0$ and $\theta_2 = \frac{d - \sigma}{d - 1}$, and using $N_3 \leq N_1$ we have
\[ | \Lambda (F,  F, v_{\ast},  v) | \begin{ltae}
\lesssim N_1^{| O (\epsilon)|  + \tilde{\epsilon}  - 2S - \frac{d - \sigma}{d - 1} \frac{\sigma - 1}{2} +  \frac{d - \sigma}{2} }
\\
\leq  N_1^{| O (\epsilon) | + \tilde{\epsilon} -  S - \frac{d-\sigma}{d-1}\frac{\sigma - 1}{2} + \frac{d - \sigma}{2} - \frac{\sigma - 1}{2} + \max\{- S + \frac{\sigma - 1}{2}, 0\}} 
\end{ltae} \]
and the assertion follows from the previous case. 
 
\textbf{Case 3b.} Suppose $(h_1, h_2, h_3, h_4) \in \{( v_*, F,  v, F), (F, v_*,  v, F)\}$, that is, $h_3 = v$. 
Since, we can interchange $N_1$ and $N_2$, we  only treat 
$( v_*, F,  v, F)$.
We apply \eqref{eq:YX} with parameter $\theta \in [0, 1]$ to be determined below and use $N_1 \approx N_2$
to have
\[ 
  | \Lambda (v_{\ast}, F, v, F) | \begin{ltae}
\lesssim N_1^{| O (\epsilon) |  + \xs - S - \theta \frac{\sigma - 1}{2} - \frac{\sigma - 1}{2}} N_3^{- \xs + \theta \frac{d - 1}{2}} N_4^{- S + \frac{\sigma - 1}{2}}
\end{ltae} 
\]
Set $\theta = \frac{d - \sigma}{d - 1}$ and since $N_3, N_4 \leq N_1$ we have that 
\[ 
  | \Lambda (v_{\ast}, F, F, v) | 
  \lesssim 
  N_1^{| O (\epsilon) |+ \tilde{\epsilon} -  S - \frac{(d-\sigma)}{(d-1)} \frac{\sigma - 1}{2} + \frac{d - \sigma}{2} - \frac{\sigma - 1}{2} + \max\{- S + \frac{\sigma - 1}{2}, 0\}},
\]
and we conclude as in Case 3a.

\textbf{Case 4.} Assume
\[
(h_1, h_2, h_3, h_4) \in \{(v_{\ast}, F, F, F), (F, v_{\ast},  F, F), (F, F, v_{\ast},  F), (F, F, F, v_{\ast})\},\]
that is, there is no $v$ in the product. Again,
since we can interchange $N_1$ and $N_2$, the first two options are equivalent, and we show that the latter two follow from the first one.

It holds that
\[ | \Lambda (v_{\ast}, F, F, F) | \begin{ltae}
\lesssim N_1^{| O (\tilde{\epsilon}) | + \xs - \frac{\sigma - 1}{2}} N_2^{- S - \frac{\sigma - 1}{2}} N_3^{- S + \frac{\sigma - 1}{2}} N_4^{- S + \frac{\sigma - 1}{2}}
\\
\lesssim N_1^{| O (\tilde{\epsilon}) | + \tilde{\epsilon} - S - (\sigma - 1) + \xs - 2 \min \Big( S - \frac{\sigma - 1}{2}, 0 \Big)} \lesssim N_1^{- \tilde{\epsilon}} \,,
\end{ltae} \]
where the last inequality holds  as long as
\[
S + 2 \min \Big( S - \frac{\sigma - 1}{2}, 0 \Big) > \frac{d -   \sigma}{2} - (\sigma - 1) + 2 \tilde{\epsilon} + | O (\epsilon) |,
\]
and, in particular, if
\[ \begin{cases}
S > \frac{d - \sigma}{6} + 2 \tilde{\epsilon} + | O (\epsilon) |,
\\
S > \big( \frac{d + 1}{2} - \sigma \big) - \frac{\sigma - 1}{2} + 2
\tilde{\epsilon} + | O (\epsilon) | .
\end{cases} \]
Also,
\[ | \Lambda (F, F, v_{\ast},  F) | \begin{ltae}
\lesssim N_1^{| O (\tilde{\epsilon}) | - S - \frac{\sigma - 1}{2}} N_2^{- S - \frac{\sigma - 1}{2}} N_3^{\xs + \frac{\sigma - 1}{2}} N_4^{- S + \frac{\sigma - 1}{2}}
\\
\lesssim N_1^{| O (\tilde{\epsilon}) | + \tilde{\epsilon} - 2S - \frac{\sigma - 1}{2} + \xs -  \min \Big( S - \frac{\sigma -  1}{2}, 0 \Big)}
\\
\leq N_1^{| O (\tilde{\epsilon}) | + \tilde{\epsilon} - S - (\sigma - 1) + \xs - 2 \min \Big( S - \frac{\sigma - 1}{2}, 0 \Big)} \,,
\end{ltae} \]
and the assertion follows from the previous one.   Moreover,
\[ | \Lambda (F, F,  F, v_{\ast}) | \begin{ltae}
\lesssim N_1^{| O (\tilde{\epsilon}) | - S - \frac{\sigma - 1}{2}} N_2^{- S - \frac{\sigma - 1}{2}} N_3^{-S + \frac{\sigma - 1}{2}} N_4^{\xs + \frac{\sigma - 1}{2}}
\\
\lesssim N_1^{| O (\tilde{\epsilon}) | + \tilde{\epsilon} - 3S  + \xs }
\\
\leq N_1^{| O (\tilde{\epsilon}) | + \tilde{\epsilon} - S - (\sigma - 1) + \xs - 2 \min \Big( S - \frac{\sigma -       1}{2}, 0 \Big)}  \,,
\end{ltae} \]
and the assertion follows as above.

Let us  summarize the restrictions on $S$.

\begin{table*}[h]
\begin{tabular}{|p{2.0cm}|p{7.0cm}|}
\hline
Case 1. & $S > 2 \tilde{\epsilon} + | O (\epsilon)
|$                                                       \\
\hline
Case 2. & $\begin{cases}
S > 2 \tilde{\epsilon} + | O (\tilde{\epsilon}) | \\
S > \Big( \frac{d - \sigma}{d - 1} \Big) \Big( \frac{d + 1}{2}
- \sigma \Big) + 2 \tilde{\epsilon} + | O (\tilde{\epsilon})
|
\end{cases}$                  \\
\hline
Case 3. & $\begin{cases}
S > \frac{d - \sigma}{2}  -\frac{d-\sigma}{d-1}\frac{\sigma - 1}{2}  - \frac{\sigma - 1}{2}
+ 2 \tilde{\epsilon} + | O
(\epsilon) | \\
S > \frac{(d - \sigma)^2}{4 (d - 1)} + 2 \tilde{\epsilon} + | O
(\epsilon) |
\end{cases}$                  \\
\hline
Case 4. & $\begin{cases}
S > \frac{d - \sigma}{6} + 2 \tilde{\epsilon} + | O (\epsilon)
|, \\
S > \Big( \frac{d + 1}{2} - \sigma \Big) - \frac{\sigma - 1}{2} + 2
\tilde{\epsilon} + | O (\epsilon) |
\end{cases}$                  \\
\hline
\end{tabular}
\end{table*}

We claim that the restrictions $\frac{d - \sigma}{6}$ and  $\Big( \frac{d - \sigma}{d - 1}
  \Big) \Big( \frac{d + 1}{2} - \sigma \Big)$ are the strongest ones for any small $\epsilon, \tilde{\epsilon} > 0$. Indeed,  
\[
\Big( \frac{d - \sigma}{d - 1}  \Big) \Big( \frac{d + 1}{2} - \sigma \Big) \geq \frac{d - \sigma}{2}  -\frac{d-\sigma}{d-1}\frac{\sigma - 1}{2}  - \frac{\sigma - 1}{2}
\]
  and
\[
  \Big( \frac{d - \sigma}{d - 1}  \Big) \Big( \frac{d + 1}{2} - \sigma \Big) \geq  \Big( \frac{d + 1}{2} - \sigma \Big) - \frac{\sigma - 1}{2} 
\]
are equivalent after standard algebraic manipulations to $\sigma \geq 1$. Also, 
\[
\frac{(d - \sigma)^2}{4 (d - 1)}
\begin{ltae}
= \frac{3}{4} \frac{d - \sigma}{6} +  \frac{1}{4} \Big( \Big( \frac{d - \sigma}{d - 1}  \Big) \Big( \frac{d + 1}{2} - \sigma \Big)  \Big)
\\
\leq  \max \Big\{ \frac{d - \sigma}{6},    \Big( \frac{d - \sigma}{d - 1}  \Big) \Big( \frac{d + 1}{2} - \sigma \Big)  \Big\} \,,
\end{ltae}
\]
 as claimed. 
 
To compare the two largest bounds, we note that  $\frac{d - \sigma}{6} \geq \Big( \frac{d - \sigma}{d - 1}\Big) \Big( \frac{(d + 1)}{2} - \sigma \Big)$ if and only if  $\sigma \geq \frac{d+ 2}{3}$ and the assertion follows. 
\end{proof*}

\section{Almost sure local well-posedness of cubic
  NLS}\label{sec:almost-sure-well-posedness}

This section is devoted to the proof of Theorem \ref{thm:main}, that is, to
the local almost sure well-posedness of the cubic NLS.
Theorem \ref{thm:main} is an immediate
consequence of the local well-posedness for a forced cubic equation and the
bound on $\|e^{- i t P} f^{\omega} \|_{Y^{S, \epsilon} (I)}$, established in
\Cref{prop:Y-bounds}. Specifically, we consider the problem
\begin{equation}\label{eq:finaleq}
\begin{cases}
(i \partial_t - \Lap) v = \pm |F + v|^2 (F + v),
\\
v (0) = 0
\end{cases} 
\end{equation}
for some $F : \R \times \R^d \rightarrow \mathbb{C}$ such that
$\|F\|_{Y^{S, \epsilon} (I)} < \infty$. We have the following local well-posedness
result.

\begin{proposition} \label{prop:exlocforced} Let $I$ be an open time interval
containing $0$ such that $| I | < T_0$ for some small fixed $T_0 > 0$
depending only on $\Lap$ and $d$ defined in Lemma
\ref{lem:strichartz}.  Fix $S, \epsilon > 0$.  Then, there exists
$0 < \delta \ll 1$ such that if $F \in Y^{S, \epsilon} (I)$ satisfies
\[
\|F\|_{Y^{S, \epsilon} (I)} \leq \delta,
\]
then there exists a unique solution
\[
v \in C \big( I ; H^{\frac{d - \sigma}{2}} (\R^d) \big) \cap X^{\xs,  \epsilon} (I)
\]
to \eqref{eq:finaleq} on $I \times \R^d$.
\end{proposition}

\begin{proof}
Fix small $\delta > 0$ determined below, and
use a fixed point argument to construct our solution. We define
\[
\mathcal{B}_{\delta} = \Big\{v \in X (I) : \|v\|_{X^{\xs, \epsilon} (I)} \leq 2 \delta\Big\}
\]
and the map
\[
\Phi (v) (t) = \mp \int_0^t e^{- i (t - s) \Lap} |F + v|^2 (F + v) (s) \dd s,
\]
which, by Duhamel's formula, is a solution of \eqref{eq:finaleq}. Next,
\Cref{prop:main-linear-estimate} and Corollary \ref{cor:main-nonlinear-estimate} imply for any sufficiently small $\epsilon, \tilde{\epsilon}, \delta > 0$ and any  $v \in \mathcal{B}_{\delta}$,
\[
\| \Phi (v)\|_{X^{\xs, \epsilon} (I)} \lesssim \||F + v|^2 (F + v)\|_{X^{\xs + \tilde{\epsilon}, \epsilon} (I)^*}
\lesssim \|v\|_{X^{\xs, \epsilon} (I)}^3 +\|F\|_{Y^{S, \epsilon} (I)}^3 \leq 2 \delta \,,
\]
and in addition, for any $v_1, v_2 \in \mathcal{B}$,
\[
\begin{ltae}
\| \Phi (v_1) - \Phi (v_2)\|_{X^{\xs, \epsilon} (I)}
\\ \qquad
\begin{ltae}
\lesssim  \Big\| |F + v_1 |^2 (F + v_1) - |F + v_2 |^2 (F + v_2)\Big\|_{X^{\xs + \tilde{\epsilon}, \epsilon} (I)^*}
\\
\lesssim  \|v_1 - v_2 \|_{X^{\xs, \epsilon} (I)} \Big(\|v_1 \|^2_{X^{\xs , \epsilon} (I)} + \|v_2 \|_{X^{\xs, \epsilon} (I)}^2 +\|F\|_{Y^{S, \epsilon} (I)}^2\Big)
\\       
\lesssim \dfrac{\|v_1 - v_2 \|_{X^{\xs, \epsilon} (I)}}{2} .
\end{ltae}
\end{ltae}
\]
So $\Phi : \mathcal{B}_{\delta} \rightarrow \mathcal{B}_{\delta}$ is a contraction with respect to the $X^{\xs, \epsilon} (I)$ norm for sufficiently small $\epsilon > 0$. Thus, there exists a unique solution to \eqref{eq:finaleq}.
\end{proof}

Next, we state a continuity property for the norms $X^{\xs ,\epsilon}(I)$ and $Y^{S,\epsilon}(I)$ with respect to the interval which follows from the dominated convergence theorem.

\begin{lemma}
\label{lem:norm-continuity-t}
Let $I\subset \R$ be a closed interval. Fix $S,\epsilon>0$. Assume that $\|v\|_{X^{\xs ,\epsilon} (I)}<\infty $  and $\|F\|_{Y^{S,\epsilon}(I)}<\infty$. Then the mappings
\[
t\in I \rightarrow \|v\|_{X^{\xs ,\epsilon}\big([\inf I,t]\big)},\qquad t\in I \rightarrow \|F\|_{Y^{S,\epsilon}\big([\inf I,t]\big)}
\]
and
\[
t\in I \rightarrow \|v\|_{X^{\xs ,\epsilon}\big([t,\sup I]\big)},\qquad t\in I \rightarrow \|F\|_{Y^{S,\epsilon}\big([t,\sup I]\big)}
\]
are continuous. We can also allow half-open and open intervals $I$.
\end{lemma}

We are now in position to prove Theorem \eqref{thm:main}.

\begin{proof}[Proof of Theorem \ref{thm:main}]
We are looking for a solution to \eqref{eq:eqintro} of the form
\[ u (t) = e^{- i t \Lap} f^{\omega} + v (t), \]
where $v$ is a solution to
\begin{equation}
\begin{cases}
(i \partial_t - \Lap) v = \pm \big|e^{i t \Lap} f^{\omega} + v\big|^2 (e^{- i t \Lap} f^{\omega} + v) & \text{ on } I \times \R^d,
\\
v (0) = 0.
\end{cases}
\label{eq:eqf2}
\end{equation}
Since by \Cref{prop:Y-bounds} one has $\big\|e^{i t \Lap} f^{\omega} \big\|_{Y^{S, \epsilon} (\R)}
< \infty$, for a.e. $\omega \in \Omega$, then by Lemma \ref{lem:norm-continuity-t}, we
can find an interval $I^{\omega} \subset \R$ with $0 \in I^{\omega}$ such
that
\[
\|e^{- i t \Lap} f^{\omega} \|_{Y^{S, \epsilon} (I^{\omega})} \leq \delta,
\] where
$0 < \delta \ll 1$ is the constant given in Proposition \ref{prop:exlocforced},
and consequently there exists a unique solution $v \in C \big(I^{\omega} ;
\dot{H}_x^{\alpha} (\R^d)\big) \cap X^{\xs, \epsilon} (I^{\omega})$ to \eqref{eq:eqf2} for a.e.
$\omega \in \Omega$. To show global uniqueness, assume that $w$ is a
solution of \eqref{eq:eqf2} belonging to $v \in C \big( I ; H^{\frac{d - \sigma}{2}} (\R^d) \big) \cap X^{\xs, \epsilon} (I)$. Then, from the
continuity and $w (0) = 0$ follows that $w$ is small for short times, and
therefore by the uniqueness of small solutions, $w y= v$. Iterating this
procedure, we obtain uniqueness on $I^{\omega}$.
\end{proof}

{\raggedright \printbibliography}

@article{bennettSharpKplaneStrichartz2018,
  ids = {bennettSharpPlaneStrichartz2018a},
  title = {A sharp k-plane {{Strichartz}} inequality for the {{Schr\"odinger}} equation},
  author = {Bennett, Jonathan and Bez, Neal and Flock, Taryn and Guti{\'e}rrez, Susana and Iliopoulou, Marina},
  year = {2018},
  month = aug,
  journal = {Trans. Amer. Math. Soc.},
  volume = {370},
  number = {8},
  pages = {5617--5633},
  issn = {0002-9947, 1088-6850},
  doi = {10.1090/tran/7309},
  langid = {english}
}

@article{benyiHigherOrderExpansions2019,
  title = {Higher order expansions for the probabilistic local {{Cauchy}} theory of the cubic nonlinear {{Schr\"odinger}} equation on {{$\mathbb{R}$}}{$^{3}$}},
  author = {B{\'e}nyi, {\'A}rp{\'a}d and Oh, Tadahiro and Pocovnicu, Oana},
  year = {2019},
  journal = {Trans. Amer. Math. Soc. Ser. B},
  volume = {6},
  number = {4},
  pages = {114--160},
  issn = {2330-0000},
  doi = {10.1090/btran/29},
  langid = {english}
}

@article{benyiProbabilisticCauchyTheory2015,
  ids = {benyiProbabilisticCauchyTheory2014},
  author = {B{\'e}nyi, {\'A}rp{\'a}d and Oh, Tadahiro and Pocovnicu, Oana},
  year = {2015},
  journal = {Trans. Amer. Math. Soc. Ser. B},
  volume = {2},
  number = {1},
  eprint = {1405.7327},
  eprinttype = {arxiv},
  pages = {1--50},
  issn = {2330-0000},
  doi = {10.1090/btran/6},
  archiveprefix = {arXiv},
  langid = {english},
  title = {On the probabilistic {{Cauchy}} theory of the cubic nonlinear {{Schr\"odinger}} equation on {$\mathbb{R}^d$}, {$d\geq 3$}}
}

@article{bourgainInvariantMeasuresDdefocusing1996,
  author = {Bourgain, Jean},
  year = {1996},
  month = jan,
  journal = {Communications in Mathematical Physics},
  volume = {176},
  number = {2},
  pages = {421--445},
  publisher = {{Springer}},
  issn = {0010-3616, 1432-0916},
  doi = {10.1007/bf02099556},
  title = {Invariant measures for the {$2$}{{D-defocusing}} nonlinear {{Schr\"odinger}} equation}
}

@article{bourgainPeriodicNonlinearSchrodinger1994,
  ids = {bourgainPeriodicNonlinearSchrodinger1994a},
  title = {Periodic nonlinear {{Schr\"odinger}} equation and invariant measures},
  author = {Bourgain, J.},
  year = {1994},
  month = jan,
  journal = {Comm. Math. Phys.},
  volume = {166},
  number = {1},
  pages = {1--26},
  publisher = {{Springer}},
  issn = {0010-3616, 1432-0916},
  doi = {10.1007/bf02099299},
  isbn = {0010-3616}
}

@article{bourgainRefinementsStrichartzInequality1998,
  author = {Bourgain, J.},
  year = {1998},
  journal = {Internat. Math. Res. Notices},
  volume = {1998},
  number = {5},
  pages = {253},
  issn = {10737928},
  doi = {10.1155/S1073792898000191},
  title = {Refinements of {{Strichartz}}' inequality and applications to {$2$}{{D-NLS}} with critical nonlinearity}
}

@article{breretonAlmostSureLocal2018,
  title = {Almost sure local well-posedness for the supercritical quintic {{NLS}}},
  author = {Brereton, Justin},
  year = {2018},
  month = aug,
  journal = {Tunisian Journal of Mathematics},
  volume = {1},
  number = {3},
  pages = {427--453},
  publisher = {{Mathematical Sciences Publishers}},
  issn = {2576-7666},
  doi = {10.2140/tunis.2019.1.427}
}

@article{burqLongTimeDynamics2013,
  title = {Long time dynamics for the one dimensional non linear {{Schr\"odinger}} equation},
  author = {Burq, Nicolas and Thomann, Laurent and Tzvetkov, Nikolay},
  year = {2013},
  journal = {Annales de l'Institut Fourier},
  volume = {63},
  number = {6},
  pages = {2137--2198},
  issn = {1777-5310},
  doi = {10.5802/aif.2825}
}

@article{burqRandomDataCauchy2008,
  title = {Random data {{Cauchy}} theory for supercritical wave equations {{I}}: local theory},
  shorttitle = {Random data {{Cauchy}} theory for supercritical wave equations {{I}}},
  author = {Burq, Nicolas and Tzvetkov, Nikolay},
  year = {2008},
  month = sep,
  journal = {Invent. Math.},
  volume = {173},
  number = {3},
  pages = {449--475},
  issn = {1432-1297},
  doi = {10.1007/s00222-008-0124-z},
  langid = {english}
}

@article{burqRandomDataCauchy2008a,
  title = {Random data {{Cauchy}} theory for supercritical wave equations {{II}}: a global existence result},
  shorttitle = {Random data {{Cauchy}} theory for supercritical wave equations {{II}}},
  author = {Burq, Nicolas and Tzvetkov, Nikolay},
  year = {2008},
  month = sep,
  journal = {Invent. math.},
  volume = {173},
  number = {3},
  pages = {477--496},
  issn = {1432-1297},
  doi = {10.1007/s00222-008-0123-0},
  langid = {english}
}

@article{cazenaveCauchyProblemCritical1990,
  author = {Cazenave, Thierry and Weissler, Fred B.},
  year = {1990},
  month = jan,
  journal = {Nonlinear Anal.},
  volume = {14},
  number = {10},
  pages = {807--836},
  issn = {0362546X},
  doi = {10.1016/0362-546x(90)90023-a},
  langid = {english},
  title = {The cauchy problem for the critical nonlinear {{Schr\"odinger}} equation in {$H^s$}}
}

@unpublished{CFU23,
  author = {Casteras, Jean-Baptiste and Földes, Juraj and Uraltsev, Gennady},
  title  = {Higher order expansions for the probabilistic local well-posedness of a cubic nonlinear Schrödinger equation},
  note = {In preparation}
}

@book{cazenaveSemilinearSchrodingerEquations2003,
  title = {Semilinear {{Schr\"odinger}} equations},
  author = {Cazenave, Thierry},
  year = {2003},
  series = {Courant lecture notes in mathematics},
  number = {10},
  publisher = {{Courant Institute of Mathematical Sciences}},
  address = {{New York, NY}},
  isbn = {978-0-8218-3399-5},
  langid = {english}
}

@article{choRemarksDispersiveEstimates2011,
  ids = {choRemarksDispersiveEstimates2011a},
  title = {Remarks on some dispersive estimates},
  author = {Cho, Yonggeun and Ozawa, Tohru and Xia, Suxia},
  year = {2011},
  month = apr,
  journal = {CPAA},
  volume = {10},
  number = {4},
  pages = {1121--1128},
  issn = {1534-0392},
  doi = {10.3934/cpaa.2011.10.1121},
  langid = {english}
}

@article{christIllposednessNonlinearSchrodinger2003,
  title = {Ill-posedness for nonlinear {{Schrodinger}} and wave equations},
  author = {Christ, Michael and Colliander, James and Tao, Terence},
  year = {2003},
  publisher = {{arXiv}},
  doi = {10.48550/arxiv.math/0311048},
  copyright = {Assumed arXiv.org perpetual, non-exclusive license to distribute this article for submissions made before January 2004}
}

@article{christMaximalFunctionsAssociated2001,
  ids = {christMaximalFunctionsAssociated2001a},
  title = {Maximal {{Functions Associated}} to {{Filtrations}}},
  author = {Christ, Michael and Kiselev, Alexander},
  year = {2001},
  month = feb,
  journal = {Journal of Functional Analysis},
  volume = {179},
  number = {2},
  pages = {409--425},
  issn = {0022-1236},
  doi = {10.1006/jfan.2000.3687}
}

@article{collianderGlobalWellposednessScattering2008,
  ids = {collianderGlobalWellPosednessScattering2008},
  title = {Global well-posedness and scattering for the energy-critical {{Schr\"odinger}} equation in {{$\mathbb{R}$}} {\textsuperscript{3}}},
  author = {Colliander, James and Keel, Markus and Staffilani, Gigliola and Takaoka, Hideo and Tao, Terence},
  year = {2008},
  month = may,
  journal = {Ann. Math.},
  volume = {167},
  number = {3},
  pages = {767--865},
  publisher = {{Annals of Mathematics}},
  issn = {0003-486X},
  doi = {10.4007/annals.2008.167.767},
  langid = {english}
}

@article{dengInvariantGibbsMeasure2021,
  title = {Invariant {{Gibbs}} measure and global strong solutions for the {{Hartree NLS}} equation in dimension three},
  author = {Deng, Yu and Nahmod, Andrea R. and Yue, Haitian},
  year = {2021},
  month = mar,
  journal = {J. Math. Phys.},
  volume = {62},
  number = {3},
  pages = {031514},
  publisher = {{American Institute of Physics}},
  issn = {0022-2488},
  doi = {10.1063/5.0045062}
}

@article{dengTwodimensionalNonlinearSchrodinger2012,
  title = {Two-dimensional nonlinear {{Schr\"odinger}} equation with random radial data},
  author = {Deng, Yu},
  year = {2012},
  month = dec,
  journal = {Anal. PDE},
  volume = {5},
  number = {5},
  pages = {913--960},
  publisher = {{Mathematical Sciences Publishers}},
  issn = {1948-206X},
  doi = {10.2140/apde.2012.5.913}
}

@article{dinhWellposednessRegularityIllposedness2018,
  title = {On well-posedness, regularity and ill-posedness for the nonlinear fourth-order {{Schr\"odinger}} equation},
  author = {Dinh, Van Duong},
  year = {2018},
  month = sep,
  journal = {Bull. Belg. Math. Soc. Simon Stevin},
  volume = {25},
  number = {3},
  pages = {415--437},
  publisher = {{The Belgian Mathematical Society}},
  issn = {1370-1444},
  doi = {10.36045/bbms/1536631236}
}

@article{dodsonAlmostSureLocal2019,
  title = {Almost sure local well-posedness and scattering for the {{4D}} cubic nonlinear {{Schr\"odinger}} equation},
  author = {Dodson, Benjamin and L{\"u}hrmann, Jonas and Mendelson, Dana},
  year = {2019},
  month = apr,
  journal = {Advances in Mathematics},
  volume = {347},
  pages = {619--676},
  issn = {0001-8708},
  doi = {10.1016/j.aim.2019.02.001},
  langid = {english}
}

@article{duongdinhRandomDataTheory2021,
  title = {Random data theory for the cubic fourth-order nonlinear {{Schr\"odinger}} equation},
  author = {Duong Dinh, Van},
  year = {2021},
  journal = {Commun. Pure Appl. Anal.},
  volume = {20},
  number = {2},
  pages = {651--680},
  issn = {1553-5258},
  doi = {10.3934/cpaa.2020284},
  langid = {english}
}

@article{ginibreSmoothingPropertiesRetarded1992,
  title = {Smoothing properties and retarded estimates for some dispersive evolution equations},
  author = {Ginibre, J. and Velo, G.},
  year = {1992},
  month = jan,
  journal = {Comm. Math. Phys.},
  volume = {144},
  number = {1},
  pages = {163--188},
  publisher = {{Springer}},
  issn = {0010-3616, 1432-0916},
  doi = {10.1007/bf02099195}
}

@article{hadacWellposednessScatteringKPII2009,
  title = {Well-posedness and scattering for the {{KP-II}} equation in a critical space},
  author = {Hadac, Martin and Herr, Sebastian and Koch, Herbert},
  year = {2009},
  month = jun,
  journal = {Annales de l'Institut Henri Poincar\'e C},
  volume = {26},
  number = {3},
  pages = {917--941},
  issn = {0294-1449},
  doi = {10.1016/j.anihpc.2008.04.002},
  langid = {english}
}

@article{herrGlobalWellposednessEnergycritical2011,
  title = {Global well-posedness of the energy-critical nonlinear {{Schr\"odinger}} equation with small initial data in {{H1}}({{T3}})},
  author = {Herr, Sebastian and Tataru, Daniel and Tzvetkov, Nikolay},
  year = {2011},
  month = aug,
  journal = {Duke Mathematical Journal},
  volume = {159},
  number = {2},
  pages = {329--349},
  publisher = {{Duke University Press}},
  issn = {0012-7094, 1547-7398},
  doi = {10.1215/00127094-1415889}
}

@article{huangInhomogeneousOscillatoryIntegrals2017,
  title = {Inhomogeneous oscillatory integrals and global smoothing effects for dispersive equations},
  author = {Huang, Tianxiao and Huang, Shanlin and Zheng, Quan},
  year = {2017},
  month = dec,
  journal = {J. Differential Equations},
  volume = {263},
  number = {12},
  pages = {8606--8629},
  issn = {0022-0396},
  doi = {10.1016/j.jde.2017.08.053},
  langid = {english}
}

@article{ionescuLowregularitySchrodingerMaps2006,
  title = {Low-regularity {{Schr\"odinger}} maps},
  author = {Ionescu, Alexandru D. and Kenig, Carlos E.},
  year = {2006},
  month = jan,
  journal = {Differential and Integral Equations},
  volume = {19},
  number = {11},
  pages = {1271--1300},
  publisher = {{Khayyam Publishing, Inc.}},
  issn = {0893-4983}
}

@article{ionescuLowregularitySchrodingerMaps2007,
  shorttitle = {Low-regularity {{Schr\"odinger}} maps, {{II}}},
  author = {Ionescu, Alexandru D. and Kenig, Carlos E.},
  year = {2007},
  month = apr,
  journal = {Commun. Math. Phys.},
  volume = {271},
  number = {2},
  pages = {523--559},
  issn = {1432-0916},
  doi = {10.1007/s00220-006-0180-4},
  langid = {english},
  title = {Low-regularity {{Schr\"odinger}} maps, {{II}}: global well-posedness in dimensions {$d  \geq 3$}}
}

@article{karpmanStabilitySolitonsDescribed2000,
  title = {Stability of solitons described by nonlinear {{Schr\"odinger-type}} equations with higher-order dispersion},
  author = {Karpman, V.I and Shagalov, A.G},
  year = {2000},
  month = sep,
  journal = {Physica D: Nonlinear Phenomena},
  volume = {144},
  number = {1-2},
  pages = {194--210},
  issn = {01672789},
  doi = {10.1016/s0167-2789(00)00078-6},
  langid = {english}
}

@article{karpmanStabilizationSolitonInstabilities1996,
  title = {Stabilization of soliton instabilities by higher-order dispersion: {{Fourth-order}} nonlinear {{Schr\"odinger-type}} equations},
  shorttitle = {Stabilization of soliton instabilities by higher-order dispersion},
  author = {Karpman, V. I.},
  year = {1996},
  month = feb,
  journal = {Phys. Rev. E},
  volume = {53},
  number = {2},
  pages = {R1336-R1339},
  issn = {1063-651X, 1095-3787},
  doi = {10.1103/physrevE.53.R1336},
  langid = {english}
}

@article{keelEndpointStrichartzEstimates1998,
  title = {Endpoint {{Strichartz}} estimates},
  author = {Keel, Markus and Tao, Terence},
  year = {1998},
  journal = {Amer. J. Math.},
  volume = {120},
  number = {5},
  pages = {955--980},
  publisher = {{Johns Hopkins University Press}},
  issn = {1080-6377},
  doi = {10.1353/ajm.1998.0039}
}

@book{kochDispersiveEquationsNonlinear2014,
  title = {Dispersive equations and nonlinear waves: generalized {{Korteweg-de Vries}}, nonlinear {{Schr\"odinger}}, wave and {{Schr\"odinger}} maps},
  shorttitle = {Dispersive equations and nonlinear waves},
  author = {Koch, Herbert and Tataru, Daniel and Vi{\c s}an, Monica},
  year = {2014},
  series = {Oberwolfach seminars},
  number = {45},
  publisher = {{Birkh\"auser}},
  address = {{Basel}},
  isbn = {978-3-0348-0736-4 978-3-0348-0735-7},
  langid = {english}
}

@article{lebowitzStatisticalMechanicsNonlinear1988,
  title = {Statistical mechanics of the nonlinear {{Schr\"odinger}} equation},
  author = {Lebowitz, Joel L. and Rose, Harvey A. and Speer, Eugene R.},
  year = {1988},
  month = feb,
  journal = {J Stat Phys},
  volume = {50},
  number = {3-4},
  pages = {657--687},
  issn = {0022-4715, 1572-9613},
  doi = {10.1007/bf01026495},
  langid = {english}
}

@article{mckeanStatisticalMechanicsNonlinear1995,
  title = {Statistical mechanics of nonlinear wave equations. {{IV}}. {{Cubic Schr\"odinger}}},
  author = {McKean, H. P.},
  year = {1995},
  month = jan,
  journal = {Communications in Mathematical Physics},
  volume = {168},
  number = {3},
  pages = {479--491},
  publisher = {{Springer}},
  issn = {0010-3616, 1432-0916},
  doi = {10.1007/BF02101840}
}

@article{nahmodNonlinearSchrodingerEquation2015,
  title = {The nonlinear {{Schr\"odinger}} equation on tori: {{Integrating}} harmonic analysis, geometry, and probability},
  shorttitle = {The nonlinear {{Schr\"odinger}} equation on tori},
  author = {Nahmod, Andrea R.},
  year = {2015},
  month = sep,
  journal = {Bull. Amer. Math. Soc.},
  volume = {53},
  number = {1},
  pages = {57--91},
  issn = {0273-0979, 1088-9485},
  doi = {10.1090/bull/1516},
  langid = {english}
}

@article{ozawaSpacetimeEstimatesNull1998,
  title = {Space-time estimates for null gauge forms and nonlinear {{Schr\"odinger}} equations},
  author = {Ozawa, T. and Tsutsumi, Y.},
  year = {1998},
  month = jan,
  journal = {Differential Integral Equations},
  volume = {11},
  number = {2},
  issn = {0893-4983},
  doi = {10.57262/die/1367341068}
}

@article{pausaderGlobalWellposednessEnergy2007,
  title = {Global well-posedness for energy critical fourth-order {{Schr\"odinger}} equations in the radial case},
  author = {Pausader, Benoit},
  year = {2007},
  journal = {Dyn. Partial Differ. Equ.},
  volume = {4},
  number = {3},
  pages = {197--225},
  issn = {1548159X, 21637873},
  doi = {10.4310/dpde.2007.v4.n3.a1},
  langid = {english}
}

@article{pausaderMasscriticalFourthorderSchrodinger2010,
  title = {The mass-critical fourth-order {{Schr\"odinger}} equation in high dimensions},
  author = {Pausader, Benoit and Shao, Shuanglin},
  year = {2010},
  month = dec,
  journal = {J. Hyper. Differential Equations},
  volume = {07},
  number = {04},
  pages = {651--705},
  issn = {0219-8916, 1793-6993},
  doi = {10.1142/s0219891610002256},
  langid = {english}
}

@article{ryckmanGlobalWellposednessScattering2007,
  title = {Global well-posedness and scattering for the defocusing energy-critical nonlinear {{Schr\"odinger}} equation in {{R}} 1+4},
  author = {Ryckman, E and Visan, M},
  year = {2007},
  journal = {American Journal of Mathematics},
  volume = {129},
  number = {1},
  pages = {1--60},
  issn = {1080-6377},
  doi = {10.1353/ajm.2007.0004},
  langid = {english}
}

@article{shenAlmostSureWellPosedness2023,
  title = {Almost {{Sure Well-Posedness}} and {{Scattering}} of the {{3D Cubic Nonlinear Schr\"odinger Equation}}},
  author = {Shen, Jia and Soffer, Avy and Wu, Yifei},
  year = {2023},
  month = jan,
  journal = {Commun. Math. Phys.},
  volume = {397},
  number = {2},
  pages = {547--605},
  issn = {1432-0916},
  doi = {10.1007/s00220-022-04500-z},
  langid = {english}
}

@article{spitzAlmostSureLocal2023,
  title = {Almost sure local wellposedness and scattering for the energy-critical cubic nonlinear {{Schr\"odinger}} equation with supercritical data},
  author = {Spitz, Martin},
  year = {2023},
  month = apr,
  journal = {Nonlinear Analysis},
  volume = {229},
  pages = {113204},
  issn = {0362-546X},
  doi = {10.1016/j.na.2022.113204},
  langid = {english}
}

@article{syAlmostSureGlobal2021,
  title = {Almost sure global well-posedness for the energy supercritical {{Schr\"odinger}} equations},
  author = {Sy, Mouhamadou},
  year = {2021},
  month = oct,
  journal = {Journal de Math\'ematiques Pures et Appliqu\'ees},
  volume = {154},
  pages = {108--145},
  issn = {00217824},
  doi = {10.1016/j.matpur.2021.08.002},
  langid = {english}
}

@article{thomannRandomDataCauchy2009,
  title = {Random data {{Cauchy}} problem for supercritical {{Schr\"odinger}} equations},
  author = {Thomann, Laurent},
  year = {2009},
  month = dec,
  journal = {Annales de l'Institut Henri Poincar\'e C},
  volume = {26},
  number = {6},
  pages = {2385--2402},
  issn = {0294-1449},
  doi = {10.1016/j.anihpc.2009.06.001},
  langid = {english}
}

@article{tzvetkovConstructionGibbsMeasure2009,
  title = {Construction of a {{Gibbs}} measure associated to the periodic {{Benjamin}}\textendash{{Ono}} equation},
  author = {Tzvetkov, N.},
  year = {2009},
  month = jan,
  journal = {Probab. Theory Relat. Fields},
  volume = {146},
  number = {3},
  pages = {481},
  issn = {1432-2064},
  doi = {10.1007/s00440-008-0197-z},
  langid = {english}
}

@article{zhangRandomDataCauchy2012,
  title = {Random {{Data Cauchy Theory}} for the {{Generalized Incompressible Navier}}\textendash{{Stokes Equations}}},
  author = {Zhang, Ting and Fang, Daoyuan},
  year = {2012},
  month = jun,
  journal = {J. Math. Fluid Mech.},
  volume = {14},
  number = {2},
  pages = {311--324},
  issn = {1422-6928, 1422-6952},
  doi = {10.1007/s00021-011-0069-7},
  langid = {english}
}

\end{document}